\documentclass[12pt]{extarticle}
\usepackage[utf8]{inputenc}
\usepackage[T1]{fontenc}

\usepackage{caption}
\oddsidemargin \evensidemargin
\usepackage{amssymb}
\usepackage{amsmath}
\usepackage{comment}
\usepackage{amsthm}
\usepackage{amssymb}
\usepackage{algorithmicx}
\usepackage{algorithm}
\usepackage{algpseudocode}
\usepackage{mathtools}
\usepackage[margin=1cm]{caption}
\usepackage{subcaption}
\usepackage{mathrsfs}
\usepackage{xcolor}
\usepackage{tikz,tikz-3dplot,tikz-cd}
\usepackage{pgfplots}
\pgfplotsset{compat=1.16}
\usepackage{url}
\usepackage[finalizecache=false,frozencache=true]{minted}
\setminted[julia]{frame=lines,
rulecolor=\color{white!80!black},
fontsize=\small,
numbers=right,
numbersep=-5pt,
obeytabs=true,
tabsize=4}

\definecolor{mycolor1}{rgb}{0.00000,0.44700,0.74100}
\definecolor{mycolor2}{rgb}{0.8500, 0.3250, 0.0980}
\definecolor{mycolor3}{rgb}{0.9290, 0.6940, 0.1250}
\definecolor{mycolor4}{rgb}{0.4940, 0.1840, 0.5560}
\definecolor{mycolor5}{rgb}{0.4660, 0.6740, 0.1880}

\newtheorem{assumption}{Assumption}
\newtheorem{definition}{Definition}[section]
\newtheorem{proposition}{Proposition}[section]
\newtheorem{corollary}{Corollary}[section]
\newtheorem{theorem}{Theorem}[section]

\theoremstyle{remark}

\newtheorem{examplex}{Example}[section]
\newenvironment{example}
  {\pushQED{\qed}\examplex}
  {\popQED\endexamplex}

\usepackage{mathtools} 
\mathtoolsset{showonlyrefs,showmanualtags}
\numberwithin{equation}{section}

\newcommand{\R}{\mathbb{R}}
\newcommand{\Z}{\mathbb{Z}}
\newcommand{\C}{\mathbb{C}}
\newcommand{\N}{\mathbb{N}}
\newcommand{\Q}{\mathbb{Q}}
\newcommand{\PP}{\mathbb{P}}
\newcommand{\A}{\mathcal{A}}
\newcommand{\B}{\mathcal{B}}
\newcommand{\CC}{\mathcal{C}}

\title{Contour Integration for Eigenvector Nonlinearities}
\author{Rob Claes, Karl Meerbergen and Simon Telen}
\date{}

\begin{document}

\maketitle

\begin{abstract}
    Solving polynomial eigenvalue problems with eigenvector nonlinearities (PEPv) is an interesting computational challenge, outside the reach of the well-developed methods for nonlinear eigenvalue problems. 
    We present a natural generalization of these methods which leads to a contour integration approach for computing all eigenvalues of a PEPv in a compact region of the complex plane. Our methods can be used to solve any suitably generic system of polynomial or rational function equations.
\end{abstract}

\section{Introduction}

We consider a matrix valued function $T : \C^n \times \C \rightarrow \C^{n \times n}, (x,z) \mapsto T(x,z)$ such that, for any fixed $z \in \C$, $T$ is given by homogeneous polynomials in $x$, and for any fixed $x$, $T$ is given by polynomials in $z$. We assume moreover that all polynomials in the $i$-th row of $T$ are of the same degree $d_i$. If any of these degrees is positive, the function $T$ defines a polynomial eigenvalue problem with eigenvector nonlinearities (PEPv), given by the equations
\begin{equation} \label{eq:PEPv}
     T(x,z) \cdot x = 0. 
\end{equation}
By homogeneity, these equations are well-defined on $\PP^{n-1} \times \C$, where $\PP^{n-1}$ is the $(n-1)$-dimensional complex projective space. Points $(x^*, z^*) \in \PP^{n-1} \times \C$ such that $T(x^*,z^*)\cdot x^* =0$ are called eigenpairs. For such an eigenpair, $z^*$ is the eigenvalue, with corresponding eigenvector $x^*$. This paper is concerned with computing all eigenpairs $(x^*, z^*)$ for which $z^*$ lies in a compact domain $\Omega \subset \C$, whose Euclidean boundary is denoted by $\partial \Omega$.

\begin{example}[$n = 3, d_1 = d_2 = d_3 = 1$] \label{ex:intro}
Consider the PEPv given by 
\[ T(x,z) \cdot x = \begin{pmatrix}
x_1 + z x_2 & z x_2+x_3 & x_1-x_3 \\
x_1+(1+z)x_2 & (1-z^2)x_2-zx_3 & x_1 +x_3\\ 
(1+z)x_1 +x_2 & x_2-x_3 & zx_1+(1-z)x_3
\end{pmatrix} \cdot \begin{pmatrix} x_1 \\ x_2 \\ x_3
\end{pmatrix} = \begin{pmatrix} 0 \\ 0 \\ 0
\end{pmatrix}.\]
For fixed $z \in \C$, the rows define three conics in the projective plane $\PP^2$. Usually, these three conics have no common intersection points. The eigenvalues $z = z^*$ are precisely those choices of $z$ for which the three conics intersect. The 12 eigenvalues are the roots of 
\[{\cal R}(z) = 4z^{12}+12z^{11}-z^{10}-53z^9-100z^8-108z^7-78z^6-23z^5+14z^4+22z^3+8z^2-4z+3,\]
depicted in Figure~\ref{fig:intro_sols}.
For instance, $z^* \approx 0.5919$ is an eigenvalue, with eigenvector $x^* \approx (1 : -1.9218 : -1.9646) \in \PP^2$. 
A possible choice for the target domain $\Omega$ to select this eigenvalue is shown in Figure~\ref{fig:intro_sols} by its boundary $\partial\Omega$.
The three conics corresponding to $z^* = 0.5919$ are shown in Figure \ref{fig:intro}.
\end{example}

\begin{figure}[tb]
    \begin{subfigure}[b]{.48\textwidth}
        \centering
        \begin{tikzpicture}[/tikz/background rectangle/.style={fill={rgb,1:red,1.0;green,1.0;blue,1.0}, draw opacity={1.0}},scale = 0.92]
\begin{axis}[point meta max={nan}, point meta min={nan}, legend cell align={left}, legend columns={1}, title={}, title style={at={{(0.5,1)}}, anchor={south}, font={{\fontsize{14 pt}{18.2 pt}\selectfont}}, color={rgb,1:red,0.0;green,0.0;blue,0.0}, draw opacity={1.0}, rotate={0.0}}, legend style={color={rgb,1:red,0.0;green,0.0;blue,0.0}, draw opacity={1.0}, line width={1}, solid, fill={rgb,1:red,1.0;green,1.0;blue,1.0}, fill opacity={1.0}, text opacity={1.0}, font={{\fontsize{8 pt}{10.4 pt}\selectfont}}, text={rgb,1:red,0.0;green,0.0;blue,0.0}, cells={anchor={center}}, at={(1.02, 1)}, anchor={north west}}, axis background/.style={fill={rgb,1:red,1.0;green,1.0;blue,1.0}, opacity={1.0}}, anchor={north west}, xshift={1.0mm}, yshift={-1.0mm}, width={70mm}, height={70mm}, scaled x ticks={false}, xlabel={$\Re(z)$}, x tick style={color={rgb,1:red,0.0;green,0.0;blue,0.0}, opacity={1.0}}, x tick label style={color={rgb,1:red,0.0;green,0.0;blue,0.0}, opacity={1.0}, rotate={0}}, xlabel style={at={(ticklabel cs:0.5)}, anchor=near ticklabel, at={{(ticklabel cs:0.5)}}, anchor={near ticklabel}, font={{\fontsize{11 pt}{14.3 pt}\selectfont}}, color={rgb,1:red,0.0;green,0.0;blue,0.0}, draw opacity={1.0}, rotate={0.0}},  xmin={-1.6801588660066602}, xmax={2.5029040585803934}, xtick={{-1.0,0.0,1.0,2.0}}, xticklabels={{$-1$,$0$,$1$,$2$}}, xtick align={inside}, xticklabel style={font={{\fontsize{8 pt}{10.4 pt}\selectfont}}, color={rgb,1:red,0.0;green,0.0;blue,0.0}, draw opacity={1.0}, rotate={0.0}}, x grid style={color={rgb,1:red,0.0;green,0.0;blue,0.0}, draw opacity={0.1}, line width={0.5}, solid}, axis x line*={left}, x axis line style={color={rgb,1:red,0.0;green,0.0;blue,0.0}, draw opacity={1.0}, line width={1}, solid}, scaled y ticks={false}, ylabel={$\Im(z)$}, y tick style={color={rgb,1:red,0.0;green,0.0;blue,0.0}, opacity={1.0}}, y tick label style={color={rgb,1:red,0.0;green,0.0;blue,0.0}, opacity={1.0}, rotate={0}}, ylabel style={at={(ticklabel cs:0.5)}, anchor=near ticklabel, at={{(ticklabel cs:0.5)}}, anchor={near ticklabel}, font={{\fontsize{11 pt}{14.3 pt}\selectfont}}, color={rgb,1:red,0.0;green,0.0;blue,0.0}, draw opacity={1.0}, rotate={0.0}}, ymin={-1.4416526225849546}, ymax={1.4416526225849546}, ytick={{-1.0,0.0,1.0}}, yticklabels={{$-1$,$0$,$1$}}, ytick align={inside}, yticklabel style={font={{\fontsize{8 pt}{10.4 pt}\selectfont}}, color={rgb,1:red,0.0;green,0.0;blue,0.0}, draw opacity={1.0}, rotate={0.0}}, y grid style={color={rgb,1:red,0.0;green,0.0;blue,0.0}, draw opacity={0.1}, line width={0.5}, solid}, axis y line*={left}, y axis line style={color={rgb,1:red,0.0;green,0.0;blue,0.0}, draw opacity={1.0}, line width={1}, solid}, colorbar={false}]
    \addplot[color={rgb,1:red,0.0;green,0.6056;blue,0.9787},  only marks, draw opacity={1.0}, line width={0}, solid, mark={*}, mark size={1.50 pt}, mark repeat={1}, mark options={color=mycolor1, draw opacity={1.0}, fill=mycolor1, fill opacity={1.0}, line width={0.75}, rotate={0}, solid}]
        table[row sep={\\}]
        {
            \\
            -0.5166001009052071  -0.828656618389808  \\
            -1.1538313760548315  -0.2182072653957205  \\
            -0.5166001009052065  0.8286566183898081  \\
            0.5919305292430951  -6.130680781421075e-46  \\
            0.03266379320803539  0.9631352535118162  \\
            -1.5617702926692907  -0.7683088286685945  \\
            0.2113149691782338  0.314616665230113  \\
            0.03266379320803556  -0.9631352535118162  \\
            -1.56177029266929  0.7683088286685948  \\
            -1.1538313760548318  0.21820726539572027  \\
            0.21131496917823395  -0.31461666523011306  \\
            2.384515485243024  -1.1479437019748901e-41  \\
        }
        ;
        \label{marker:intro_eigenvalues}
    \addplot[color=mycolor2, draw opacity={1.0}, line width={1}, solid]
        table[row sep={\\}]
        {
            \\
            1.0  0.0  \\
            0.9967160055292985  0.038363148505351795  \\
            0.9869179452156118  0.0760963751728522  \\
            0.9707667029384087  0.11258010146381223  \\
            0.9485274816493557  0.14721526560118134  \\
            0.9205654487471826  0.1794331591473648  \\
            0.887339740039091  0.20870476518104591  \\
            0.8493959207434935  0.23454944474040892  \\
            0.8073570273242101  0.2565428289016038  \\
            0.7619133372489575  0.27432378690474374  \\
            0.7138110346524129  0.28760035591099814  \\
            0.6638399580133517  0.29615453502433503  \\
            0.6128206310286621  0.29984586486020637  \\
            0.5615907896369272  0.2986137338847595  \\
            0.5109916264174742  0.2924783736545471  \\
            0.461853978231477  0.2815405266149281  \\
            0.41498468390366594  0.2659797919119  \\
            0.3711533359511322  0.2460516763790868  \\
            0.33107964389547323  0.22208339912259464  \\
            0.2954216166523462  0.19446851859233652  \\
            0.2647647580432637  0.1636604703631646  \\
            0.23961245283903232  0.13016512173526745  \\
            0.22037770119573252  0.09453246540708626  \\
            0.2073763372035738  0.05734758861041163  \\
            0.20082184289986543  0.019221065994213967  \\
            0.20082184289986543  -0.019221065994213898  \\
            0.20737633720357385  -0.057347588610411684  \\
            0.2203777011957324  -0.09453246540708606  \\
            0.23961245283903232  -0.1301651217352674  \\
            0.26476475804326366  -0.16366047036316453  \\
            0.2954216166523461  -0.19446851859233646  \\
            0.3310796438954733  -0.22208339912259467  \\
            0.371153335951132  -0.24605167637908668  \\
            0.41498468390366583  -0.26597979191189997  \\
            0.46185397823147667  -0.28154052661492807  \\
            0.5109916264174741  -0.2924783736545471  \\
            0.5615907896369273  -0.2986137338847595  \\
            0.6128206310286619  -0.29984586486020637  \\
            0.6638399580133517  -0.29615453502433503  \\
            0.7138110346524128  -0.28760035591099825  \\
            0.7619133372489575  -0.27432378690474374  \\
            0.8073570273242101  -0.2565428289016038  \\
            0.8493959207434933  -0.23454944474040895  \\
            0.8873397400390908  -0.20870476518104608  \\
            0.9205654487471826  -0.17943315914736485  \\
            0.9485274816493557  -0.14721526560118134  \\
            0.9707667029384088  -0.11258010146381212  \\
            0.9869179452156117  -0.07609637517285225  \\
            0.9967160055292984  -0.03836314850535199  \\
            1.0  -7.347880794884119e-17  \\
        }
        ;
    \label{marker:intro_contour}
\end{axis}
\end{tikzpicture}
        \caption{Eigenvalues (\ref{marker:intro_eigenvalues}) and contour $\partial\Omega$ (\ref{marker:intro_contour}).}
        \label{fig:intro_sols}
    \end{subfigure}
    \begin{subfigure}[b]{.48\textwidth}
        \centering
        \input{conics_ex1_1}
        \caption{Three conics corresponding to $z^*\approx0.5919$.}
        \label{fig:intro}
    \end{subfigure}
    \caption{Example \ref{ex:intro}.}
\end{figure}

Any system of polynomial equations $f_1(x,z) = \cdots = f_n(x,z) = 0$ on $\PP^{n-1} \times \C$ can be formulated as a PEPv. Rewriting this as in \eqref{eq:PEPv} and calling solutions `eigenpairs' seemingly does not change much. Our motivation is that the algorithm we propose for finding eigenpairs with $z^* \in \Omega$ is a natural generalization of standard algorithms used for eigenvalue problems with more structure. More precisely, PEPv's generalize polynomial eigenvalue problems (PEP), for which $d_i = 0$. These in turn contain generalized eigenvalue problems (GEP), for which $d_i = 0$ and $T(z) = A - z \cdot B$ is an affine-linear function. 

Polynomial eigenvalue problems often arise from an intermediate step in solving general nonlinear eigenvalue problems (NEP), in which the entries of $T(z)$ are allowed to be transcendental functions of $z$.
One typically approximates these functions by polynomials in a certain region of the complex plane, obtaining a PEP.
One way of solving PEPs is linearization \cite{effenberger2012chebyshev,van2015linearization}. The linearization step results in a GEP of larger dimension.
This dimension grows with the degree of the approximating polynomials, and is typically very large.
In order to solve it, special structure exploiting methods are used \cite{jarlebring2012linear,van2015compact}.

Another common approach for solving NEPs is based on contour integration.
The goal of methods like Beyn \cite{beyn2012integral}, SS \cite{asakura2009numerical} or NLFEAST \cite{gavin2018feast} is to locate all eigenvalues on a compact domain $\Omega$ in the complex plane.
This is done by calculating a contour integral over the boundary $\partial \Omega$ with an integrand that contains the matrix inverse of the eigenvalue problem.
Using the residue theorem, the poles of the integrand -- which coincide with the desired eigenvalues in the compact domain -- can be extracted.

In the present paper, we develop a new contour-integration-based method for finding all eigenpairs of a PEPv with $z^* \in \Omega$. It generalizes known approaches for PEPs, in the sense that when $d_i = 0$, Beyn's algorithm is recovered. We reiterate that, under suitable genericity assumptions, this can be used to find all solutions to a polynomial system $f_1(x,z) = \cdots = f_n(x,z) = 0$ with $z$-coordinate inside $\Omega$. The situation of interest is where the number of solutions with this property is much smaller than the total number of solutions, i.e., the total number of eigenvalues of $T(x,z)$. Our strategy is to integrate trace functions along the boundary $\partial \Omega$, and extract the eigenvalues from moments. These traces are evaluated using numerical homotopy continuation \cite{sommese2005numerical}. Such methods can also be used to naively compute all eigenpairs of $T(x,z)$ and then filter out relevant solutions by checking whether $z \in \Omega$. However, an important feature of our method is that evaluating the trace usually requires significantly less homotopy paths than the total number of eigenvalues of $T(x,z)$, which makes it more efficient than the naive approach.
It is important to note that the traces are not available in an explicit form as is usually expected for PEPs solved by Krylov methods.
Therefore, we only consider contour integration methods in this paper: these only require evaluation of the trace, not its explicit expression.

This paper is structured as follows.
An overview of the standard Beyn's algorithm is presented in Section~\ref{sec:2}.
The basis of our approach is laid in Section~\ref{sec:3} by introducing the concepts of resultants and traces. Section~\ref{sec:4} describes the resulting contour integration method and comments on the numerical implementation.
We discuss the complexity of our method in Section~\ref{sec:5} and present an analysis for two families of systems of equations. Our numerical experiments in Section~\ref{sec:6} confirm the presented theory.

\section{Beyn's algorithm} \label{sec:2}

The method of Beyn \cite{beyn2012integral} considers the nonlinear eigenvalue problem defined by the holomorphic matrix valued function $A:\C\rightarrow\C^{n\times n}$ as
\[ A(z) \cdot x = 0.\]
The goal is to find eigenpairs $(x^*,z^*)\in \PP^{n-1}\times\C$ for which the eigenvalue $z^*$ lies in the compact domain $\Omega$ of the complex plane. The function $A$ is typically assumed to be holomorphic in a neighborhood of $\Omega$.
Beyn's method is especially useful for targeting a specific subset of the, possibly infinite, complete set of eigenvalues.
In this section, we recapitulate the idea and theory behind contour integration for eigenvalue problems.
For reasons of clarity, we focus the derivations on simple eigenvalues only.
An eigenvalue is called simple if the algebraic multiplicity and the geometric multiplicity are equal to one, where the multiplicity of an eigenvalue is defined by the following definitions.
\begin{definition}
    The algebraic multiplicity of an eigenvalue $z^*$ is the smallest positive integer $m_a$ such that
    \begin{equation}
        \left.\frac{d^{m_a}}{dz^{m_a}}\det ( A(z) ) \right\vert_{z=z^*}\neq 0.
    \end{equation}
\end{definition}
\begin{definition}
    The geometric multiplicity of an eigenvalue $z^*$ is the dimension of the null space of $A(z^*)$.
\end{definition}

Let $z^*$ be a simple eigenvalue of $A$ with corresponding right and left eigenvectors $x^*$ and $y^*$ such that $A(z^*)\cdot x^*=0$ and $A(z^*)^H\cdot y^*=0$.
There exists a region $\mathcal{N}\subset\C$ around $z^*$ and a holomorphic function $R: \C \rightarrow \C^{n\times n}$ such that
\[ A(z)^{-1} = \frac{1}{z-z^*}x^*y^{*H} + R(z) , \quad z\in \mathcal{N}\setminus\{z^*\}. \]
This property can be easily generalized to the case where multiple simple eigenvalues are considered in a compact subset of $\C$ \cite[Thm.~2.4]{beyn2012integral}.
\begin{theorem}\label{thm:keldysh_simple}
Let $\Omega \subset\C$ be a compact subset that contains only the simple eigenvalues $z^*_i, i=1,\ldots,l$ with corresponding right and left eigenvectors $x^*_i$ and $y^*_i$.
Then there exists a neighborhood $\mathcal{N}$ of $\Omega$ and a holomorphic function $R: \C \rightarrow\C^{n \times n}$ such that
\[ A(z)^{-1} = \sum_{i=1}^l \frac{1}{z-z^*_i} \, x^*_iy^{*H}_i + R(z), \quad z\in\mathcal{N}\setminus\{z^*_1,\ldots,z^*_l\}.\]
\end{theorem}

Theorem~\ref{thm:keldysh_simple} provides us with a way of expressing the value of a contour integral over the boundary of the compact subset $\Omega\subset\mathcal{N}$.
\begin{theorem} \label{thm:res_simple}
In the situation of Theorem \ref{thm:keldysh_simple}, we have that 
\[ \frac{1}{2\pi\sqrt{-1}}\oint_{\partial\Omega} f(z)A(z)^{-1}dz = \sum_{i=1}^l f(z_i^*)  \,x_i y_i^H.\]
\end{theorem}

Under the assumption that only a few eigenvalues lie within $\Omega$, i.e., $l<n$, and all eigenvectors are linearly independent, we can extract the eigenvalues and corresponding eigenvectors from the following two contour integrals
\[ A_0 = \frac{1}{2\pi\sqrt{-1}}\oint_{\partial\Omega} A(z)^{-1}\hat{V}dz ,
\qquad  A_1 = \frac{1}{2\pi\sqrt{-1}}\oint_{\partial\Omega} zA(z)^{-1}\hat{V}dz,\]
with $\hat{V}\in\C^{n\times q}$, $q\geq l$ a random matrix of full rank $q$. 
Using Theorem~\ref{thm:res_simple}, we see that
\[A_0 = \sum_{i=1}^l x_i^*y_i^{*H}\hat{V} = XY^H\hat{V}, \quad A_1 = \sum_{i=1}^lz_i^*x_i^*y_i^{*H}\hat{V} = XZY^H\hat{V},\]
where $X$ and $Y$ have the right and left eigenvectors for their columns and $Z$ is a diagonal matrix containing the corresponding eigenvalues.
The matrix $A_0$ has rank at most $l$ for random choices of $\hat{V}$, so that a reduced singular value decomposition can be expressed as
\[ A_0 = V_0\Sigma_0W_0^H\]
with rectangular $V_0\in\C^{n\times l}$ and $W_0\in\C^{q\times l}$ and diagonal matrix $\Sigma_0 = \text{diag}(\sigma_1,\ldots,\sigma_l)$.
In \cite{beyn2012integral} it is shown, via some linear algebra manipulations, that 
\[ V_0^HA_1W_0\Sigma_0^{-1} = SZS^{-1}.\]
This decomposition reveals the diagonal matrix $Z$ containing the eigenvalues, while the corresponding eigenvectors can be extracted from $V=V_0S$.

Since nonlinear eigenvalue problems can have more eigenvalues than the size of the matrix, it is necessary to extend this approach to the case where $l>n$.
Luckily, Beyn's algorithm generalizes easily to this case.
First the matrix $\hat{V}\in\C^{n\times n}$ is now a square matrix of full rank which is used to calculate so-called higher order moments of the contour integrals:
\[ A_k = \frac{1}{2\pi\sqrt{-1}}\oint_{\partial\Omega} z^kA(z)^{-1}\hat{V}dz.\]
It should be clear that $A_k$ can be decomposed as $A_k = XZ^kY^H\hat{V}$.
From these higher order moments, we can calculate two block Hankel matrices 
\begin{equation} \label{eq:Bi}
    B_0=\begin{pmatrix}
A_0 &\cdots& A_{M-1} \\
\vdots & & \vdots \\
A_{M-1} & \cdots &A_{2M-2}
\end{pmatrix}, \text{ and }
B_1=\begin{pmatrix}
A_1 & \cdots & A_{M} \\
\vdots & & \vdots \\
A_M & \cdots & A_{2M-1}
\end{pmatrix}. 
\end{equation}
In a similar way as with few eigenvalues, it can be shown that the rank of $B_0$ is equal to the number of eigenvalues in $\Omega$ such that the diagonazible matrix
\[ V_0^HB_1W_0\Sigma_0^{-1} = SZS^{-1}\]
is defined by the reduced singular value decomposition $B_0 = V_0\Sigma_0W_0^H$.
The eigenvalues are again the elements of the diagonal matrix $Z$ while the corresponding eigenvectors can be extracted from the first $n$ rows of $V_0S$.
Some additional technicalities need to be considered in the case of semi-simple and defective eigenvalues \cite{beyn2012integral}, but this falls outside the scope of this discussion.

We conclude the section with a discussion on how the moment matrices $A_k$ are computed in practice. We assume that $\partial\Omega$ is parameterized by a continuous function $\varphi: [0,2\pi) \rightarrow \C$.
The moment matrix $A_k$ is then expressed as
\[A_k = \frac{1}{2\pi\sqrt{-1}} \int_0^{2\pi} \varphi^k(t)A(\varphi(t))^{-1}\hat{V}\varphi^\prime(t) dt.\]
This integral can be approximated numerically by the trapezoidal rule with $N$ equidistant points $t_\ell=\frac{2 \ell \pi}{N}, \ell =0,\ldots,N-1$ as
\[ A_k \approx A_{k,N} = \frac{1}{N\sqrt{-1}}\sum_{\ell=0}^{N-1} \varphi^k(t_\ell)A(\varphi(t_\ell))^{-1}\hat{V}\varphi^\prime(t_\ell).\]
The choice of the trapezoidal rule integration scheme with equidistant points might feel somewhat arbitrary, but it often leads to satisfactory results with a limited amount of points \cite{beyn2012integral}. 
The impact of the integration scheme on the accuracy of the results is discussed in \cite{van2016nonlinear}.

The largest part of the computational cost of Beyn's method originates from the calculation of the moment matrices.
Note that most of the computation work can be reused between every moment matrix since the factor $A(\varphi(t_\ell))^{-1}\hat{V}$ is independent of the moment index $k$.
Each linear system $A(\varphi(t_\ell))^{-1}\hat{V}$ can be solved independently for every value of $t_\ell$ which leads to an efficient parallel implementation. In what follows, our aim is to generalize Beyn's method to the case with eigenvector nonlinearities.

\section{Resultants and traces} \label{sec:3}

In this section, we turn back to the PEPv from the Introduction. We discuss resultants and traces related to our equations $T(x,z) \cdot x = 0$. These algebraic objects fit into our strategy for solving a PEPv as follows.
\begin{enumerate}
    \item There is a polynomial ${\cal R}(z)$, obtained by evaluating a \emph{resultant}, whose roots are the eigenvalues of $T(x,z)$.
    \item \emph{Traces} are rational functions in $z$ whose denominator is (roughly) ${\cal R}(z)$.
    \item Traces can be evaluated using tools from numerical nonlinear algebra. This allows to perform numerical \emph{contour integration} along $\partial \Omega$ to compute eigenvalues.
\end{enumerate}
This section addresses points 1 and 2. Point 3 is the subject of the next section. We work in the ring $K[x] = K[x_1, \ldots, x_n]$ of polynomials in the variables $x_i$ with coefficients in the rational function field $K = \C(z)$. The polynomials $f_1, \ldots, f_n \in K[x]$ are the entries of the vector $T(x,z) \cdot x$. We assume that $f_i$ is homogeneous of degree $d_i + 1$ and write $f_i \in K[x_i]_{d_i+1}$.
\subsection{Resultants}
For fixed values $z = z^*$, the system of polynomial equations $f_1 = \cdots = f_n = 0$ encoded by the PEPv $T(x,z^*) \cdot x = 0$ consists of $n$ homogeneous equations on $\PP^{n-1}$. Generically, one expects such equations to have no solution with nonzero coordinates. The eigenvalues~are those special values of $z^*$ for which they \emph{do} have solutions, see Example \ref{ex:intro}. This is captured by a polynomial ${\cal R}(z)$ obtained via \emph{resultants}. We summarize the basics, and refer the reader to \cite[Chapters 3 and 7]{cox2006using} for more details.
Let $\A_i \subset \N^n, i = 1, \ldots, n$ denote the \emph{supports} of the polynomials $f_i \in K[x]$: if $f_i = \sum_{\alpha \in \N^n} c_{i,\alpha}(z) \, x^\alpha$, where $x^\alpha$ is short for $x_1^{\alpha_1} \cdots x^{\alpha_n}$, then 
\[ \A_i = \{ \alpha \in \N^n ~|~ c_{i,\alpha} \neq 0 \}. \]
We write $K[x]_{d_i+1} \supset K[x]_{\A_i} \simeq K^{|\A_i|}$ for the affine space over $K$ of polynomials with support contained in $\A_i$. A natural set of coordinates for $K[x]_{\A_i}$ is given by the coefficients $\{b_{i,\alpha} ~|~ \alpha \in \A_i\}$ of a generic polynomial with support $\A_i$: $h_i = \sum_{\alpha \in \A_i} b_{i,\alpha} x^\alpha \in K[x]_{\A_i}$. 
Let $Z_0 \subset K[x]_{\A_1} \times \cdots \times K[x]_{\A_n}$ be the set of tuples $(h_1, \ldots, h_n)$ for which $h_1 = \cdots = h_n = 0$ has a solution in $(K \setminus \{ 0 \} )^n$.
Its Zariski closure is $Z = \overline{Z_0} \subset K[x]_{\A_1} \times \cdots \times K[x]_{\A_n}$. Under mild assumptions on the ${\cal A}_i$, $Z$ has codimension one, so that it is defined by one polynomial equation in the coefficients of $h_1, \ldots, h_n$ \cite[Cor.~1.1]{sturmfels1994newton}. It turns out that, in this case, $Z$ is an irreducible variety defined over $\Q$ \cite[Lem.~1.1]{sturmfels1994newton}. 
The \emph{sparse resultant} $R_{\A_1, \ldots, \A_n}$ is the unique (up to sign) irreducible polynomial in $\Z[\, b_{i,\alpha}~|~ i =1, \ldots, n, \alpha \in \A_i \,]$ such that 
\[ (h_0, \ldots, h_n) \in Z \quad \Longleftrightarrow \quad R_{\A_1, \ldots, \A_n}(h_0, \ldots, h_n) = 0. \]
Evaluating the sparse resultant $R_{\A_1, \ldots, \A_n}$ at our tuple $(f_1, \ldots, f_n)$ means plugging in the coefficients $c_{i,\alpha}(z) \in K$ for the $b_{i,\alpha}$. Since we assume the coefficients of the $f_i$ to be polynomials in $z$, we obtain a polynomial 
\begin{equation} \label{eq:calR}
    {\cal R}(z) = R_{\A_1, \ldots, \A_n} (f_1, \ldots, f_n) \quad \in \, \C[z]. 
\end{equation} 
\begin{example}
Let $\A = \A_1 = \A_2 = \A_3 \subset \Z^3$ consist of all monomials of degree 2 in 3 variables. Consider 3 general ternary quadrics
\[ h_i \, = \, b_{i,1} \, x_1^2 + b_{i,2} \,  x_2^2 + b_{i,3} \, x_3^2 + b_{i,4} \, x_1x_2 + b_{i,5} \, x_1x_3 + b_{i,6} \, x_2x_3, \quad i = 1, 2, 3.\]
The resultant $R_{\A,\A,\A}$ is a polynomial of degree 12 in the 18 variables $b_{i,j}, i = 1, \ldots, 3, j = 1, \ldots, 6$, which characterizes when the three conics $\{h_i = 0 \} \subset \PP^2$ intersect. It has 21894 terms and can be computed as a $ 6 \times 6$ determinant, see \cite[Chapter 3, \S2]{cox2006using}. Plugging in the coefficients, i.e.~$b_{1,1} = 1, b_{1,2} = z, b_{1,3} = -1, b_{1,4} = z, \ldots$, we obtain the polynomial ${\cal R}(z) = R_{\A,\A,\A}(f_1,f_2,f_3)$ shown in Example \ref{ex:intro}.
\end{example}
\begin{example}
In the case of a polynomial eigenvalue problem (PEP) given by $T(z) \cdot x = 0$, we have ${\cal R}(z) = \det T(z)$.
\end{example}
\begin{definition}
The PEPv given by $T(x,z) \cdot x = 0$ is called \emph{regular} if ${\cal R}(z) \neq 0$.
\end{definition}
Unlike in the case of PEPs, regularity of a PEPv does not mean that there are finitely many eigenvalues. Here is an example.
\begin{example} \label{ex:regular}
We consider the PEPv $T(x,z) \cdot x = 0$ where 
\[ T(x,z) = \begin{pmatrix}
x_1 & (1+z)x_1 & x_2 \\ 
2x_1 & 3x_1 & (3+z)x_2 \\
2z x_1 & x_1 & x_2
\end{pmatrix} \quad \text{and} \quad \begin{matrix}
f_1 = x_1^2+(1+z)x_1x_2 + x_2x_3, \\
f_2 = 2x_1^2 + 3x_1x_2 + (3+z)x_2x_3,\\
f_3 = 2zx_1^2 + x_1x_2 + x_2x_3.
\end{matrix}\]
We calculate ${\cal R}(z) = 2z^3+8z^2-3z \neq 0$. However, for any $z^* \in \C$, $T(x^*,z^*) \cdot x^* = 0$, with $x^* = (0,0,1)^\top$ or $x^* = (0,1,0)^\top$.
\end{example}

To avoid such artefacts, we will limit ourselves to computing eigenpairs $(z^*, x^*)$ for which $x^*$ has no zero coordinates. That is, we look for eigenvectors in the \emph{algebraic torus} $\{x \in \PP^{n-1} ~|~ x_i \neq 0, i = 1, \ldots, n \}$. For such an eigenpair, we say that $z^*$ is an \emph{eigenvalue with toric eigenvector}. 
By construction, if $z^* \in \C$ is an eigenvalue of $T(x,z)$ with toric eigenvector, then ${\cal R}(z^*) = 0$. This implies the following statement. 
\begin{theorem}
A regular PEPv has finitely many eigenvalues with toric eigenvector. 
\end{theorem}
It is \emph{not} true in general that each $z^*$ such that ${\cal R}(z^*) =0$, is an eigenvalue with toric eigenvector. We continue Example \ref{ex:regular}.

\begin{example}
There are no toric solutions to $T(x,z^*) \cdot x = 0$, with $z^* = 0$ and $T$ as in Example \ref{ex:regular}. This eigenvalue is picked up by our polynomial ${\cal R}(z)$ because it corresponds to a solution of $T(x,z^*) \cdot x = 0$ in a \emph{toric compactification} of $(\C \setminus \{0\})^n$. Note that for this eigenvalue, there is an `extra' non-toric eigenvector $(0,1,-1)^\top$.
\end{example}

\begin{definition}
An eigenvalue of the PEPv $T(x,z) \cdot x = 0$ with toric eigenvector is called \emph{simple} if it is a simple zero of ${\cal R}(z)$.
\end{definition}

\begin{example}
In Example \ref{ex:intro}, $z^* \approx 0.5919$ is a simple eigenvalue with toric eigenvector.
\end{example}
\subsection{Traces}
The roots of the polynomial ${\cal R}(z)$ are eigenvalues of the PEPv given by $T(x,z)\cdot x$. It is usually hard to compute ${\cal R}(z)$. In this section we discuss rational functions in $z$, called \emph{traces}, whose denominator is ${\cal R}(z)$. The upshot is that these traces can be evaluated using tools from numerical nonlinear algebra, so that residue techniques can be used to approximate its poles. 
We fix $n$ random homogeneous polynomials $a_1, \ldots, a_n \in \C[x]$ such that $\deg(a_i) = d_i = \deg(f_i)-1$. We write $a_i \in \C[x]_{d_i}$ and collect them in a vector $a = (a_1, \ldots, a_n)^\top \in \C[x]^n$. Consider the ideal $I_a$ generated by the entries of $T(x,z) \cdot x - a$:
\begin{equation} \label{eq:Ia} 
I_a = \langle f_1 - a_1, \ldots, f_n - a_n \rangle \subset K[x, x^{-1}]. \end{equation}
Here $K[x, x^{-1}] = K[x_1^{\pm 1}, \ldots, x_n^{\pm 1}]$ is the Laurent polynomial ring in $n$ variables with coefficients in $K$. Note that the ideal $I_a$ is \emph{not} homogeneous. We will assume throughout that the equations $f_i - a_i = 0$ have finitely many solutions in $(\overline{K} \setminus \{0\})^n$, where $\overline{K}$ is the algebraic closure of $K$. This is the field of Puiseux series $\overline{K} = \C \{ \! \{ z \} \! \} $. By \cite[Ch.~5, \S 3, Thm.~6]{cox2013ideals}, our assumption can equivalently be phrased as follows.
\begin{assumption} \label{assum:finitedim}
The dimension $\delta = \dim_K K[x, x^{-1}]/I_a$ is finite. 
\end{assumption}
The set of solutions to $f_1 - a_1 = \cdots = f_n - a_n = 0$ is denoted by 
\[ V(I_a) = \{ \xi \in (\overline{K} \setminus \{0\})^n ~|~ f_i(\xi) - a_i = 0, i = 1, \ldots, n \}. \]
A point $\xi \in V(I_a)$ has multiplicity $\mu(\xi)$. By Assumption \ref{assum:finitedim}, $\sum_{\xi \in V(I_a)} \mu(\xi) = \delta$. 
\begin{definition} \label{def:trace}
For a polynomial $p \in K[x,x^{-1}]$, the \emph{trace} ${\rm Tr}_p(I_a)$  is $ \sum_{\xi \in V(I_a)} \mu(\xi) \,  p(\xi)$.
\end{definition}
\begin{proposition} \label{prop:traceisrat}
For any Laurent polynomial $p \in K[x,x^{-1}]$, the trace ${\rm Tr}_p(I_a)$ is a rational function in $z$. That is, ${\rm Tr}_p(I_a) \in K$.
\end{proposition}
\begin{proof}
This is a standard result from Galois theory, see for instance \cite[Ch.~6, Thm.~1.2]{lang02}. Another way to see this explicitly is by considering the $K$-linear map 
\[ M_p : K[x,x^{-1}]/I_a \longrightarrow K[x,x^{-1}]/I_a \quad \text{given by} \quad [f] \longmapsto [pf],\]
where $[f]$ denotes the residue class of $f \in K[x,x^{-1}]$ in $K[x,x^{-1}]/I_a$. This is called a \emph{multiplication map}. A matrix representation of such a map can be computed using linear algebra over $K$. A standard algorithm uses Gr\"obner bases \cite[Ch.~2, \S 4]{cox2006using}. Since $M_p$ can be represented by a $\delta \times \delta$ matrix with entries in $K$, its trace ${\rm tr}(M_p)$ lies manifestly in $K$. Moreover, since the trace is the sum of the eigenvalues, \cite[Ch.~4, \S  2, Prop.~2.7]{cox2006using} gives ${\rm tr}(M_p) = {\rm Tr}_p(I_a)$.
\end{proof}
\begin{example} \label{ex:trace}
Let $T$ be as in Example \ref{ex:intro}. The number $\delta$ is the number of Puiseux series solutions $x(z) = (x_1(z), x_2(z), x_3(z))$ to $f_1-a_1 = f_2 -a_2 = f_3-a_3 = 0$, with
\begin{align*}
    f_1 - a_1 &\,=\, x_1^2 +zx_2x_2 + zx_2^2+x_2x_3 + x_1x_3-x_3^2 - (b_{11} x_1 + b_{12} x_2 + b_{13} x_3), \\
    f_2 - a_2 &\,=\, x_1^2+(1+z)x_1x_2 + (1-z^2)x_2^2-zx_2x_3 + x_1x_3+x_3^2 - (b_{21} x_1 + b_{22} x_2 + b_{23} x_3), \\
    f_3 - a_3 &\, = \, (1+z)x_1^2+x_1x_2 + x_2^2-x_2x_3 + zx_1x_3 + (1-z)x_3^2 - (b_{31} x_1 + b_{32} x_2 + b_{33} x_3).
\end{align*}
Here $a_i = b_{i1}x_1 + b_{i2} x_2 + b_{i3}x_3$ are generic linear forms. Using \texttt{Maple}, we find $\delta = 8$ and \small
\[ {\rm Tr}_{x_2}(I_a) = \frac{(16b_{11} + 8b_{13})\,  z^{11}+(-4 b_{33}+ \cdots -16 b_{31}) \, z^{10} + \cdots + (-8b_{33}+\cdots+2b_{32})}{{\cal R}(z)}, \]
\normalsize
where ${\cal R}(z)$ is the polynomial from Example \ref{ex:intro}.
\end{example}
The fact that ${\cal R}(z)$ shows up as the denominator of ${\rm Tr}_{x_2}(I_a)$ in Example \ref{ex:trace} is no coincidence. To state our main result, we introduce some more notation. Let ${\cal C}_i \subset \Z^n, i = 1, \ldots, s$ be finite sets of lattice points. The sublattice of $\Z^n$ affinely generated by ${\cal C}_1, \ldots, {\cal C}_s$ is
\[ L({\cal C}_1, \ldots, {\cal C}_s) = \left \{\,  \sum_{\alpha \in {\cal C}_1} \ell_{1,\alpha} \, \alpha + \cdots + \sum_{\alpha \in {\cal C}_s} \ell_{s,\alpha} \, \alpha ~\big |~  \sum_{\alpha \in {\cal C}_i} \ell_{i, \alpha} = 0, \, \ell_{i,\alpha} \in \Z \right \}. \]
Let $\A_i$ be the support of $f_i$ and $\B_i$ that of $a_i$. We will make the following assumption.
\begin{assumption} \label{assum:lattice}
The lattice $L(\A_1, \ldots, \A_n)$ is equal to $\{ \alpha \in \Z^n ~|~ \alpha_1 + \cdots + \alpha_n = 0 \}$.
This can always be realized by a change of coordinates as long as $L(\A_1, \ldots, \A_n)$ has rank $n-1$.
\end{assumption}
We set $\A_0 = \{ e_1, \ldots, e_n \}$ with $e_i$ the $i$-th standard basis vector of $\Z^n$, $\B_0 = \{0\}$ and $\CC_i = \A_i \cup \B_i$ for $i = 0, \ldots, n$. The set $\CC_0 = \{0\} \cup \A_0$ contains all lattice points of the standard simplex in $\Z^n$. 
Note that, by Assumption \ref{assum:lattice}, $L(\CC_1, \ldots, \CC_n)$ has rank $n$. For any point $\omega$ in the dual lattice $(\Z^n)^\vee = \Z^n$ and any finite subset $\CC \subset \Z^n$, we set
\[ \CC^\omega = \{ \gamma \in \CC ~|~ \langle \omega, \gamma \rangle = \min_{\gamma' \in \CC} \, \langle \omega, \gamma' \rangle  \}. \]
Here $\langle \cdot, \cdot \rangle$ is the pairing between $\Z^n$ and its dual, i.e.~the usual dot product. For a Laurent polynomial $f = \sum_{\gamma \in \CC} c_\gamma \, x^\gamma$ supported in $\CC$, we write $f^\omega$ for the \emph{leading form} of $f$ w.r.t.~$\omega$:
\[f^\omega = \sum_{\gamma \in \CC^\omega} c_\gamma \, x^\gamma.\]
Below we use the resultant $R_{\CC_0, \CC_1, \ldots, \CC_n}$, which is a polynomial in $b_{i,\gamma}, i = 0, \ldots, n, \gamma \in \CC_i$, characterizing when $h_0 = \cdots = h_n = 0$ has a solution in $(\overline{K} \setminus \{0\})^n$, with $h_i = \sum_{\gamma \in \CC_i} b_{i,\gamma}  \, x^\gamma$. 

To give an explicit formula for the trace in terms of ${\cal R}(z)$, we will make the additional assumption that our ideal $I_a$ behaves like a \emph{generic intersection} in $(\overline{K} \setminus \{0\})^n$. To make this precise, we denote by $P_i = {\rm Conv}(\CC_i) \subset \R^n$ the \emph{Newton polytope} of $f_i - a_i$. This is the convex hull of the lattice points in $\CC_i$. The \emph{mixed volume} of $P_1, \ldots, P_n$, denoted ${\rm MV}(P_1, \ldots, P_n)$, is the generic number of solutions to a system of equations with supports $\CC_1, \ldots, \CC_n$. For definitions and examples, see for instance \cite[Sec.~5.1]{telen2020thesis}.
\begin{assumption} \label{assum:mv}
The dimension $\delta = \dim_K K[x,x^{-1}]/I_a$ equals ${\rm MV}(P_1, \ldots, P_n)$.
\end{assumption}
Assumption \ref{assum:mv} implies Assumption \ref{assum:finitedim}, so it suffices to work with Assumptions \ref{assum:lattice} and \ref{assum:mv}.
\begin{theorem} \label{thm:traceformula}
Let $T(x,z) \cdot x = (f_1, \ldots, f_n)^\top =  0$ be a PEPv satisfying Assumption \ref{assum:lattice} and let $a_i \in \C[x]_{d_i}$ be such that $I_a$ satisfies Assumption \ref{assum:mv}. Let $\CC_i$ be the support of $f_i - a_i$ and $\CC_0 = \{ 0, e_1, \ldots, e_n \}$. The PEPv given by $T(x,z)$ is regular and for $ p = \sum_{\gamma \in \CC_0} c_{0,\gamma} x^\gamma$ we have
\[ {\rm Tr}_{p} (I_a) =  \frac{{\cal Q}_{p,a}(z)}{{\cal R}(z) \cdot {\cal S}_a(z)}, \quad \text{where } {\cal Q}_{p,a}(z) = \sum_{\gamma \in \CC_0} c_{0,\gamma}  \frac{\partial R_{\CC_0, \CC_1, \ldots, \CC_n}}{\partial b_{0,\gamma}} (1,f_1-a_1, \ldots, f_n-a_n),\]
${\cal R}(z)$ is as in \eqref{eq:calR} and ${\cal S}_a(z)$ is a nonzero polynomial. \end{theorem}
\begin{proof}
Our starting point is Theorem 2.3 in \cite{d2008rational}, which expresses the trace as 
\[ {\rm Tr}_{p} (I_a) =  C \cdot \frac{{\cal Q}_{p,a}(z)}{R_{\CC_0, \ldots, \CC_n}(1, f_1-a_1, \ldots, f_n-a_n)}
\]
for a nonzero constant $C$. Proposition 2.6 in the same paper writes the denominator $R_{\CC_0, \ldots, \CC_n}(1, f_1-a_1, \ldots, f_n-a_n)$ as a product of \emph{face resultants}. More precisely, we have
\[R_{\CC_0, \ldots, \CC_n}(1, f_1-a_1, \ldots, f_n-a_n) = \prod_{\omega} R_{\CC_1^\omega, \ldots, \CC_n^\omega}((f_1-a_1)^\omega, \ldots, (f_n - a_n)^\omega)^{\delta_{\omega}}, \]
where the product ranges over the primitive inward pointing facet normals $\omega$ of the Minkowski sum $P_1 + \cdots + P_n$. The exponents $\delta_{\omega}$ are defined combinatorially from the ${\cal C}_i$ in the discussion preceeding \cite[Prop.~2.6]{d2008rational}. By Assumption \ref{assum:mv}, none of the face resultants vanishes identically. Let $\omega^* = (-1,\ldots,-1) \in (\Z^n)^\vee$. We have $\CC_i^{\omega^*} = \A_i$ and $(f_i - a_i)^{\omega^*} = f_i$, which shows that $T(x,z)$ is regular and that ${\cal R}(z)^{\delta_{\omega^*}}$ is a factor in the denominator of ${\rm Tr}_{p}(I_a)$. Assumption \ref{assum:lattice} and the fact that ${\rm Conv}(\CC_0)$ is a standard simplex imply $\delta_{\omega^*} = 1$. The theorem follows by setting 
${\cal S}_a(z) = C^{-1} \cdot \prod_{\omega \neq \omega^*} R_{\CC_1^\omega, \ldots, \CC_n^\omega}((f_1-a_1)^\omega, \ldots, (f_n - a_n)^\omega)^{\delta_{\omega}}$.
\end{proof}

\begin{example} \label{ex:extraneous1}
Consider de PEPv $T(x,z) \cdot x = 0$ given by 
\[ T(x,z) = \begin{pmatrix}
1 & z & 1 \\ 2 & 1 & z \\ x_2 & (z+1)x_3 + x_2 & 0
\end{pmatrix}. \]
This satisfies Assumptions \ref{assum:lattice} and \ref{assum:mv}. We have $\CC_0 = \{ (0,0,0), (1,0,0), (0,1,0), (0,0,1) \}$, $a_1, a_2 \in \C$ and $a_3(x) = b_{31}x_1 + b_{32} x_2 + b_{33}x_3$. The trace for $p = x_1$ is 
\begin{equation} \label{eq:extrandenom}
 {\rm Tr}_{x_1}(I_a) = \frac{b_{31} \, z^4 + (a_1+a_2-b_{32}-2b_{33})\, z^3 + \cdots + (2 a_1 + 4 a_2 + b_{31}-2 b_{32} - b_{33})}{(z^2+2z-2)(z-2)}. 
 \end{equation}
Here ${\cal R}(z) = z^2 + 2z-2$ and ${\cal S}_a(z) = z-2$ is independent of $a$. We will explain the extraneous factor ${\cal S}_a(z)$ in Example \ref{ex:extraneous2} below. 
\end{example}

\begin{corollary}
If the PEPv $T(x,z) \cdot x = (f_1, \ldots, f_n)^\top =  0$ and the ideal $I_a$ satisfy Assumptions \ref{assum:lattice} and \ref{assum:mv}, then an eigenvalue $z^*$ of $T(x,z)$ with toric eigenvector is a pole of ${\rm Tr}_p(I_a)$ if ${\cal Q}_{p,a}(z^*) \neq 0$. Moreover, simple such eigenvalues correspond to simple poles of the trace. 
\end{corollary}

In the above notation. It would be desirable to have ${\cal S}_a(z)$ equal to a nonzero constant, and ${\cal Q}_{p,a}(z^*) \neq 0$ for all simple eigenvalues of $T(x,z)$. We now discuss when this happens. Let $P = P_1 + \cdots + P_n$ be the Minkowski sum of the Newton polytopes $P_i = {\rm Conv}(\CC_i)$. In the proof of Theorem \ref{thm:traceformula} we derived 
\[
    {\cal S}_a(z) = C^{-1} \cdot \prod_{\omega \neq \omega^*} R_{\CC_1^\omega, \ldots, \CC_n^\omega}((f_1-a_1)^\omega, \ldots, (f_n - a_n)^\omega)^{\delta_{\omega}},
\]
where $\omega$ ranges over the inner facet normals to $P$. It follows from the definition of $\delta_\omega$ in \cite[Section 2]{d2008rational} that the only facet normals $\omega$ for which $\delta_{\omega} \neq 0$ are those for which $0 \notin \CC_0^\omega$. This gives a sufficient condition for ${\cal S}_a(z)  \in \C \setminus \{0\}$. 
Let $P_0 = {\rm Conv}(\CC_0)$ be the standard simplex in $\R^n$. If the monomials $x_j^{d_i + 1}, j = 1, \ldots, n$ appear in $f_i$, and $x_j^{d_i}$ appear in $a_i$, then 
\begin{equation} \label{eq:dense}
 P_i = {\rm Conv}(\CC_i) = {\rm cl}((d_i + 1) \cdot P_0 \setminus (d_i \cdot P_0)), 
\end{equation}
where ${\rm cl}(\cdot)$ denotes the Euclidean closure in $\R^n$.
\begin{theorem} \label{thm:unmixed}
Let $T(x,z) \cdot x = (f_1, \ldots, f_n)^\top =  0$ be a PEPv satisfying Assumption \ref{assum:lattice},  with $\deg(f_i) = d_i + 1$. Let $a_i \in \C[x]_{d_i}$ be such that $I_a$ satisfies Assumption \ref{assum:mv} and $P_i = {\rm Conv}(\CC_i) = {\rm cl}((d_i + 1) \cdot P_0 \setminus (d_i \cdot P_0))$. Then ${\cal S}_a(z)$ in Theorem \ref{thm:traceformula} is a nonzero complex constant. 
\end{theorem}
\begin{proof}
The theorem follows from the fact that, under the assumption \eqref{eq:dense}, the facet normals of $P = P_1 + \cdots + P_n$ are 
\[ \omega^* = (-1,\ldots, -1), ~\omega_0 = (1,\ldots,1), ~ \omega_1 = (1,0,\ldots, 0), ~\omega_2 = (0, 1, \ldots, 0), ~ \omega_n = (0,0,\ldots, 1). \]
Out of these, only for $\omega = \omega^*$ we have $0 \notin \CC_0^{\omega}$. 
\end{proof}
We present one more example of a family of PEPv's for which $S_a(z) \in \C \setminus \{0\}$. We assume that all $f_i$ are of the same degree $d+1$ and such that $x_j^{d+1}$ appears in $f_j$ for all $j$. We let $a_i = c_i \, x^\beta$ consist of one term of degree $d$, with $c_i \neq 0$. The resulting polytopes $P_i$ are all equal to a pyramid of height one over the simplex $(d+1) \cdot {\rm Conv}(e_1, \ldots, e_n)$.
\begin{theorem} \label{thm:pyramid}
Let $T(x,z) \cdot x = (f_1, \ldots, f_n)^\top =  0$ be a PEPv satisfying Assumption \ref{assum:lattice},  with $\deg(f_i) = d + 1$. Let $a_i(x) = c_i\, x^\beta \in \C[x]_{d}$ be such that $I_a$ satisfies Assumption \ref{assum:mv}. Then ${\cal S}_a(z)$ in Theorem \ref{thm:traceformula} is a nonzero complex constant. 
\end{theorem}
\begin{proof}
The polytope $P = P_1 + \cdots + P_n = n \cdot P_1$ has $n+1$ normal vectors. All of these are nonnegative, except $\omega^* = (-1, \ldots, -1)$. Therefore, only $\omega^*$ satisfies $0 \notin \CC_0^{\omega}$. 
\end{proof}
If ${\rm Conv}(\A_1) = \cdots = {\rm Conv}(\A_n)$, the argument in the proof of Theorem \ref{thm:pyramid} can be used to construct more general situations in which $P_1 = \cdots = P_n$ is a pyramid over ${\rm Conv}(A_i)$ and ${\cal S}_a(z) \in \C \setminus \{0\}$. We do not work this out explicitly. Here is an example where $S_a(z) \notin \C \setminus \{0\}$. 
\begin{example}  \label{ex:extraneous2}
The polytope $P = P_1 + P_2 + P_3$ from the PEPv in Example \ref{ex:extraneous1} is shown in Figure \ref{fig:polytope}. There are six facets. Their normal vectors $\omega_i$ in the dual lattice $(\Z^3)^\vee \simeq \Z^3$ are 
\[ \omega_1 = (0,0,1), ~ \omega_2 = -(1,0,1), ~ \omega_3 = -(1, 1, 1), ~ \omega_4 = (1,0,0), ~ \omega_5 = (1,1,1), ~ \omega_6 = (0,1,0). \]
Here $\omega^* = \omega_3$. The only other facet normal for which $0 \notin \CC_0^{\omega_i}$ is $\omega_2$. We calculate
\[ \CC_1^{\omega_2} = \CC_2^{\omega_2} = \{ (1,0,0), (0,0,1) \}, \quad \CC_3^{\omega_2} = \{(1,0,0), (0,0,1), (0,1,1), (1,1,0) \}. \]
The corresponding face equations are $f_1^{\omega_2}= f_2^{\omega_2} = f_3^{\omega_2} = 0$, with
\[ f_1^{\omega_2} = x_1 + x_3, \quad f_2^{\omega_2} = 2 x_1 + z x_3, \quad f_3^{\omega_2} = (z+1)x_2x_3 + x_1x_2 - b_{31}x_1 - b_{33}x_3. \]
These have a nontrivial solution if and only if the determinant of the linear system $f_1^{\omega_2} = f_2^{\omega_2} = 0$ vanishes. This explains $R_{\CC_1^{\omega_2}, \CC_2^{\omega_2}, \CC_3^{\omega_2}} = z-2$, which gives the extraneous factor in the denominator of \eqref{eq:extrandenom}.
\begin{figure}
\centering
\includegraphics[scale=0.3]{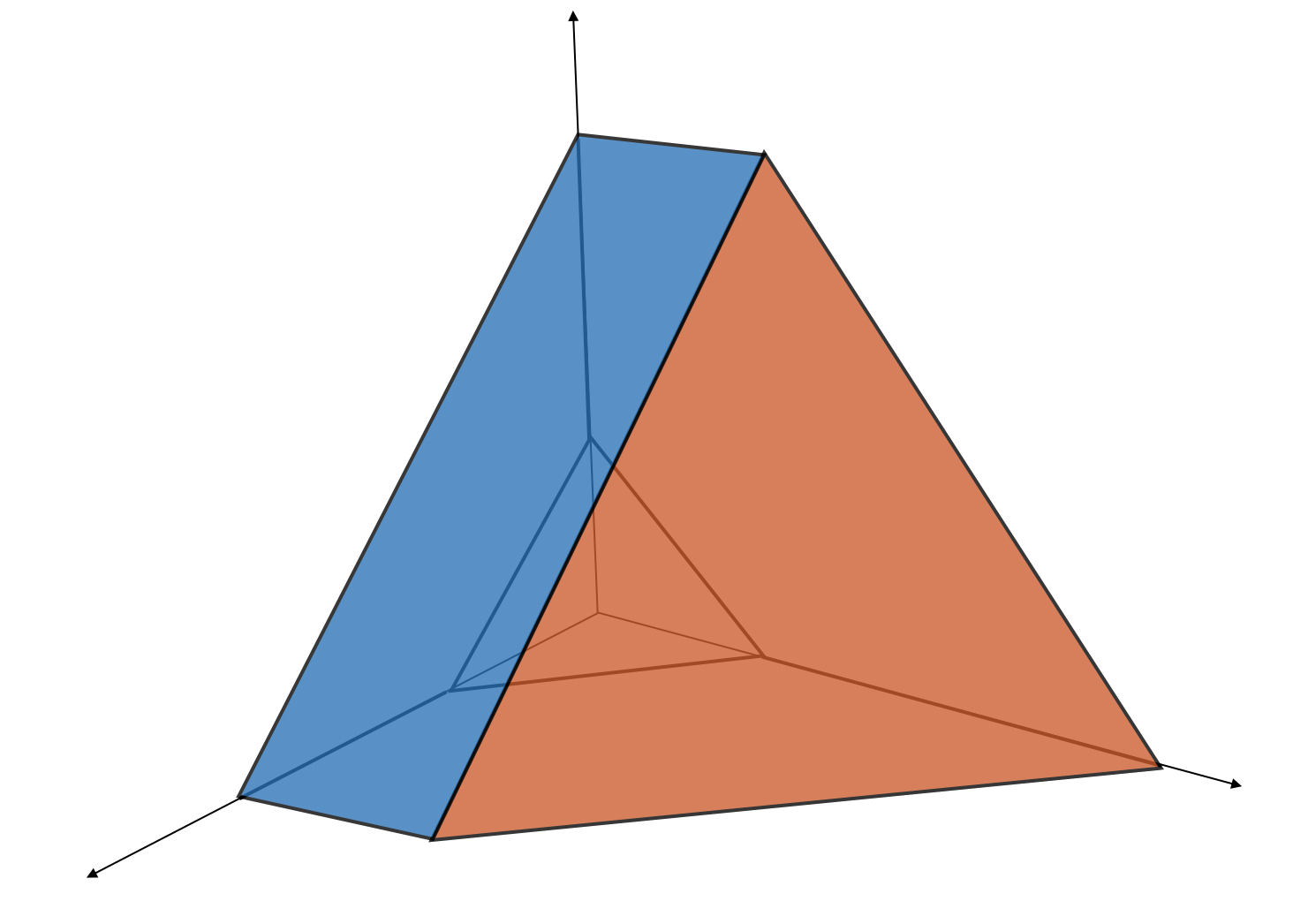}
\caption{The polytope $P$ from Example \ref{ex:extraneous2}. The facets corresponding to $\omega_2$ and $\omega_3$ are the quadrilateral and triangle coloured in blue and orange respectively. }
\label{fig:polytope}
\end{figure}
\end{example}
We conclude by briefly discussing the condition ${\cal Q}_{p,a}(z^*) \neq 0$. 
First of all, note that Assumption \ref{assum:mv} implies ${\rm Tr}_1(I_a) = \delta$, so by Theorem \ref{thm:traceformula} we have 
\[ {\cal Q}_{1,a}(z) =  \frac{\partial R_{\CC_0, \ldots, \CC_n}}{\partial b_{0,0}}(1,f_1-a_1, \ldots, f_n-a_n) = \delta \, {\cal R}(z) \, {\cal S}_a(z). \]
In particular, ${\cal Q}_{1,a}(z^*) = 0$ for every eigenvalue $z^*$ with toric eigenvector. Therefore, we will use the traces ${\rm Tr}_{x_i}(I_a)$, corresponding to the remaining exponents $\A_0 = \CC_0 \setminus \{0\}$. 

\begin{definition}
We say that an eigenvalue $z^*$ of $T(x,z)$ has a \emph{simple toric eigenvector} if ${\cal R}(z^*) = 0$ and, for generic choices of $a_i$, there is some $i \in \{1, \ldots, n\}$ for which ${\cal Q}_{x_i,a}(z^*) \neq 0$. 
\end{definition}
We point out that if $z^*$ has a simple toric eigenvector, then for generic $a_i$ the tuple $(1,f_1(x,z^*)-a_1(x), \ldots, f_n(x,z^*)-a_n(x))$ is a smooth point on the resultant hypersurface given by $\{R_{\CC_0, \ldots, \CC_n} = 0\}$. This implies that the corresponding eigenvector is unique.
We summarize the above discussion in the following theorem. 
\begin{theorem} \label{thm:main}
Under Assumptions \ref{assum:lattice} and \ref{assum:mv}, each simple eigenvalue $z^*$ of $T(x,z)$ with simple toric eigenvector is a pole of order one of the trace vector $({\rm Tr}_{x_1}(I_a), \ldots, {\rm Tr}_{x_n}(I_a)) \in \C(z)^n$. In the situations of Theorems \ref{thm:unmixed} and \ref{thm:pyramid}, all simple poles correspond to such eigenvalues.
\end{theorem}
We leave the problem of determining the precise conditions under which a simple eigenvalue has a simple toric eigenvector for future research. In our examples and experiments from Section \ref{sec:6}, we observe that this is satisfied for all simple eigenvalues.

\section{Contour integration and homotopy continuation} \label{sec:4}

Let $T(x,z)$ be a PEPv satisfying Assumptions \ref{assum:lattice} and \ref{assum:mv}. We write the trace vector from Theorem \ref{thm:main} as${\rm Tr}_{\A_0}(I_a) = ({\rm Tr}_{x_1}(I_a), \ldots, {\rm Tr}_{x_n}(I_a))$. Using Definition \ref{def:trace} and Assumption \ref{assum:mv}, we see that the entries of ${\rm Tr}_{\A_0}(I_a)$ are computed as a sum of $\delta$ terms:
\begin{equation} \label{eq:tracereminder} 
{\rm Tr}_{x_i}(I_a) = \sum_{\xi \in V(I_a)} \, \xi_i.
\end{equation} 
The simple eigenvalues with simple toric eigenvector of $T(x,z)$ are among the poles of ${\rm Tr}_{\A_0}(I_a)$. We remind the reader that $a \in \C[x]$ has homogeneous entries of degree $d_i$, where $d_i$ is the degree in $x$ of the entries in the $i$-th row of $T(x,z)$. In analogy with Beyn's method, we evaluate the trace for several vectors $a$. We collect ${\rm Tr}_{\A_0}(I_{a^{(j)}})$ for $n$ random choices $a^{(1)}, \ldots, a^{(n)} \in \C^n$ in the columns of 
\begin{equation} \label{eq:U}
    U(z) = \begin{pmatrix}
\vrule & & \vrule \\
\\
{\rm Tr}_{\A_0}(I_{a^{(1)}}) & \cdots & {\rm Tr}_{\A_0}(I_{a^{(n)}}) \\
\\
\vrule & & \vrule
\end{pmatrix} \quad \in \C(z)^{n \times n}. 
\end{equation} 
Our next result uses notation from Theorem \ref{thm:traceformula} and explains our interest in the matrix $U(z)$.
\begin{theorem} \label{thm:REP}
Let $U(z)$ be as above and let $Q(z) =( {\cal Q}_{x_i,a^{(j)}}(z))_{i,j}$. Suppose that $\det Q(z) \neq 0$ and $z^*$ is a simple eigenvalue of $T(x,z)$ with simple toric eigenvector $x^* \in \PP^{n-1}$. If ${\cal S}_{a^{(j)}}(z^*) \neq 0$ for $j = 1, \ldots, n$, we have $ U(z^*)^{-1} \cdot x^* = 0$ and $z^*$ is a simple zero of $\det U(z)^{-1}$.
\end{theorem}
\begin{proof}
If the matrix $Q(z) = ( {\cal Q}_{x_i,a^{(j)}}(z))_{i,j}$ is invertible, then so is $U(z) \in \C(z)^{n \times n }$. Indeed, Theorem \ref{thm:traceformula} implies $\det(U(z)) = \det(Q(z)) \cdot ({\cal R}(z)^n \cdot \prod_{j=1}^n {\cal S}_{a^{(j)}}(z))^{-1}$. For any $j$, we have
\[U(z) \cdot \begin{pmatrix}
0 \\ \vdots \\{\cal R}(z) \cdot {\cal S}_{a^{(j)}}(z) \\ \vdots \\ 0
\end{pmatrix} = \begin{pmatrix}
{\cal Q}_{x_1,a^{(j)}}(z) \\ \vdots \\ {\cal Q}_{x_j,a^{(j)}}(z) \\ \vdots \\ {\cal Q}_{x_n,a^{(j)}}(z)
\end{pmatrix}
\]
by Theorem \ref{thm:traceformula}. This is an equality of vectors of rational functions. We denote the right hand side by $Q_j(z)$. Left multiplying by $U(z)^{-1}$ and plugging in $z = z^*$ shows that $(z^*,Q_j(z^*))$ is an eigenpair of $U(z)^{-1}$. Here we use that $x^*$ is a simple toric eigenvector, so that $Q_j(z^*) \neq 0$. It remains to show that, as points in projective space $\PP^{n-1}$, we have $Q_j(z^*) = x^*$. For this, one adapts the proof of \cite[Lemma 3.9]{d2015poisson}. The important step requires \cite[Proposition 1.37]{d2013heights}. For brevity, we omit technicalities and leave the details to the reader. 

To see that $z^*$ is a simple zero of $\det U(z)^{-1}$, we start from the identity 
\begin{equation} \label{eq:detform}
    \det U(z)^{-1} \cdot \det Q(z) = {\cal R}^n(z) \cdot \prod_{j=1}^n {\cal S}_{a^{(j)}}(z). 
\end{equation}
We have established that $\det U(z)^{-1} = c_1 (z-z^*)^\kappa + O((z-z^*)^{\kappa+1})$ near $z = z^*$ for some $c_1 \in \C \setminus \{0\}$ and $\kappa > 0$. Moreover, since ${\cal S}_{a^{(j)}}(z^*) \neq 0$ and $z^*$ is a simple zero of ${\cal R}(z)$, the right hand side equals $c_3 (z-z^*)^n + O((z-z^*)^{n+1})$ for some $c_3 \in \C \setminus \{0\}$. Since $Q_j(z^*) = x^* \in \PP^n$ for all $j = 1, \ldots, n$, we know that ${\rm rank}(Q(z^*)) = 1$. Therefore, $(z-z^*)$ divides all but one of the invariant factors of $Q(z)$, viewed as a matrix over $\C[z]$. It follows that $\det Q(z) = c_2 (z-z^*)^\lambda + O((z-z^*)^{\lambda + 1})$ for $\lambda \geq n-1$. Since $\kappa + \lambda = n$ by \eqref{eq:detform}, we must have $\kappa = 1, \lambda = n-1$, which concludes the proof.
\end{proof}
Theorem \ref{thm:REP} shows that the matrix $U(z)$ reduces our problem to a rational eigenvalue problem of the form $U(z)^{-1} \cdot x = 0$, which can be solved using contour integration techniques from Section \ref{sec:2}. We proceed by discussing how to do this in practice. 

The $k$-th \emph{moment matrix} $A_k$ is given by 
\[ A_k \, \, = \, \,  \frac{1}{2 \pi \sqrt{-1}} \, \oint_{\partial \Omega} z^k \, U(z) \, {\rm d} z, \qquad k = 0, 1, 2, \ldots . \]
To find the poles of $U(z)$, these matrices are arranged into two block Hankel matrices $B_0, B_1$, on which we perform a sequence of standard numerical linear algebra operations. This was explained in Section \ref{sec:2}. The rank of $B_0$ equals the number of eigenvalues inside $\partial \Omega$. We emphasize that when $T(x,z) = T(z)$ represents a PEP, the matrix $U(z)$ is given by $T(z)^{-1} \cdot \begin{pmatrix}a^{(1)} & \cdots & a^{(\ell)} \end{pmatrix}^\top$ and our moment matrices $A_k$ coincide with those used in Beyn's algorithm. 
In practice, we approximate the moment matrices $A_k$ using numerical integration techniques. We assume that $\partial \Omega$ is parameterized by a differentiable map $\varphi: [0,2\pi) \rightarrow \C$, so that the $k$-th moment matrix can be written as
\[ A_k \, \, = \, \, \frac{1}{2\pi \sqrt{-1}} \, \int_0^{2\pi}U(\varphi(t)) \, \varphi^{\prime}(t)\, \varphi^k(t) \, {\rm d} t .\]
A standard approach to evaluate this integral numerically is to use the \emph{trapezoidal rule} with $N+1$ equidistant nodes $t_\ell = \frac{2\pi \ell}{N}$, $\ell=0,\ldots,N$. This gives the approximation $A_{k,N} \approx A_k$:
\begin{equation} \label{eq:int_approx} 
    A_{k,N} \, = \, \frac{1}{\sqrt{-1} \, N} \sum_{\ell=0}^{N-1}  U(\varphi(t_\ell)) \, \varphi^{\prime}(t_\ell) \, \varphi^k(t_\ell) . 
\end{equation} 
Hence, we need to evaluate $U(z)$ for $z = \varphi(t_\ell), \ell = 0, \ldots, N-1$. We do this efficiently, without explicitly constructing $U(z)$, using \emph{homotopy continuation methods}. Here, we briefly review the basics. For a complete introduction, the reader is referred to the textbook \cite{sommese2005numerical}.

For fixed $t \in [0, 2\pi)$, the trace vectors ${\rm Tr}_{{\cal A}_0}(I_{a^{(j)}})_{|z = \varphi(t)}$ are obtained by summing over the solutions to the system of polynomial equations given by $F(x,t) = 0$, where
\[ F(x,t) \, = \, T(x,\varphi(t)) \cdot x - a^{(j)}(x) \, = \, \begin{pmatrix} f_1(x,\varphi(t)) - a^{(j)}_1(x) \\ \cdots \\ f_n(x,\varphi(t)) - a^{(j)}_n(x) \end{pmatrix}. \] 
By Assumption \ref{assum:mv}, there are $\delta$ solutions. We think of these solutions as paths $x^{(m)}: [0,2 \pi) \rightarrow \C^n$ satisfying $F(x^{(m)}(t),t) = 0$, $m = 1, \ldots, \delta$. These paths are described by a system of ordinary differential equations called the \emph{Davidenko equation}:
\begin{equation} \label{eq:davidenko}
     \frac{{\rm d} F(x(t),t)}{{\rm d} t} \, = \, J_F(x(t),t) \cdot \frac{{\rm d} x}{{\rm d} t} + \frac{\partial F(x(t),t)}{\partial t} \, = \,  0 ,
\end{equation}
where $J_F$ is the Jacobian matrix whose $(j,k)$ entry is $\frac{\partial f_j}{\partial x_k}$. Each of the paths is uniquely determined by an initial condition specifying $x^{(m)}(t_0) = x^{(m)}(0)$. For computing the trace, we need to evaluate the paths at the discrete points $t_\ell = \frac{2 \pi \ell}{N}$. The situation is illustrated in Figure \ref{fig:homotopy}, where $\partial \Omega$ is the unit circle in the complex plane, parameterized by $\varphi(t) = \cos(t) + \sqrt{-1} \cdot \sin(t)$. This is drawn in orange. At each of the points $\varphi(t_\ell)$, represented as black dots on $\partial \Omega$, there are $\delta = 3$ solutions $x^{(m)}(t_\ell), m = 1, \ldots, 3$ to $F(x,t_\ell) = 0$. This is illustrated with a dashed line for one choice of $\ell$.

\begin{figure}
    \centering
    \input{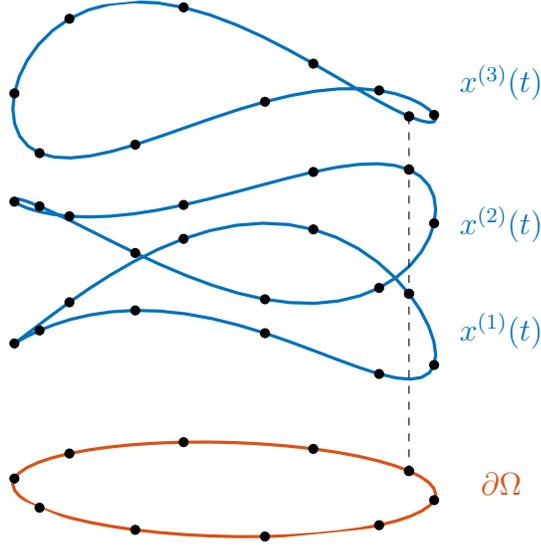}
    \caption{An illustration of the paths $x^{(m)}(t), m = 1, \ldots, \delta$ and the discretized paths $x^{(m)}(t_\ell), \ell = 0, \ldots, N$ for $\delta = 3$ and $N = 9$.}
    \label{fig:homotopy}
\end{figure}

Approximating $x^{(m)}(t_\ell)$ can be done using numerical techniques for solving the Davidenko equation \eqref{eq:davidenko}. An example is the Euler method, which approximates $x^{(m)}(t_\ell)$ from $x^{(m)}(t_{\ell-1})$ using finite differences. An important remark is that, in our scenario, we have an implicit equation $F(x(t),t) = 0$ satisfied by the solution paths. This allows us, in every step, to refine an approximation $\widetilde{x^{(m)}(t_\ell)}$ for $x^{(m)}(t_\ell)$ using \emph{Newton iteration} on $F(x,t_\ell) = 0$. With a slight abuse of notation, we also write $x^{(m)}(t_\ell)$ for the numerical approximation of $x^{(m)}(t_\ell)$ obtained \emph{after} this refinement. The path values $x^{(m)}(t_\ell)$ are used to evaluate the $i$-th column ${\rm Tr}_{{\cal A}_0}(I_{a^{(j)}})$ of $U(\varphi(t_\ell))$, by plugging $a = a^{(j)}$ and $\xi = x^{(m)}(t_\ell)$ into \eqref{eq:tracereminder}.

We summarize this discussion in Algorithm \ref{alg:homotopy} and provide some clarifying remarks. We start by pointing out that Assumption \ref{assum:mv} guarantees that for all but finitely many values $z \in \C$, the system of equations $T(x,z) \cdot x - a^{(j)}(x) = 0$ has $\delta$ isolated solutions $x \in \C^n$, each with multiplicity one. We assume that the contour $\partial \Omega$ misses these finitely many $z$-values, which makes sure that the solution paths $x^{(m)}(t)$ do not \emph{cross}, i.e.~$x^{(m)}(t) \neq x^{(m')}(t)$ for $m \neq m'$. This can be realized, if necessary, by slightly enlarging $\Omega$. In line \ref{line:start} of Algorithm \ref{alg:homotopy}, the starting points $x^{(m)}(t_0)$ are computed. This can be done using any numerical method for solving polynomial systems. Recent eigenvalue methods are described in \cite{bender2021yet}. In case of many variables, it is favorable to use the \emph{polyhedral homotopies} introduced in \cite{huber1995polyhedral}. Line \ref{line:predict} is often called the \emph{predictor} step. Our presentation assumes a first order predictor, which uses only $x^{(m)}(t_{\ell-1})$ to compute an approximation for $x^{(m)}(t_\ell)$. In practice, one sometimes uses the path values at $t_{\ell-2}, t_{\ell-3}, \ldots$ for more accurate results. It is important to remark that when $N$ is too small, the step size $2\pi/N$ may be too large to track the paths reliably. A bad approximation in line \ref{line:predict} may cause the Newton iteration in line \ref{line:correct} to converge to a \emph{different} path. This phenomenon is called \emph{path jumping}. To remedy this, one could take some `extra' steps between $t_{\ell-1}$ and $t_\ell$. Recent studies in the direction of \emph{adaptive stepsize} algorithms are \cite{telen2020robust,timme2021mixed}. Details are beyond the scope of this paper. In our implementation, the algorithm in \cite{timme2021mixed} decides how many steps to take between $t_{\ell-1}$ and $t_\ell$. Line \ref{line:correct} is called the \emph{corrector} step, and Algorithm \ref{alg:homotopy} is a blueprint for a \emph{predictor-corrector} scheme, see e.g.~\cite[Alg.~2.1]{telen2020robust}.
\begin{algorithm}[h!]
\caption{Evaluating the $j$-th column of $U(z)$ at $z = \varphi(t_\ell), \ell = 0, \ldots, N-1$}  \label{alg:homotopy}
\begin{algorithmic}[1]
\State Compute $\delta$ \emph{start solutions} $x^{(m)}(t_0), m = 1, \ldots, \delta$ satisfying $F(x^{(m)}(t_0), t_0) = 0$ \label{line:start}
\State ${\rm Tr}_{{\cal A}_0}(I_{a^{(j)}})_{|z = \varphi(t_0)} = \left (\sum_{m = 1}^\delta (x^{(m)}(t_0))_i \right )_{i = 1, \ldots, n}$
\State {$\ell \gets 0$}
\While {$\ell \leq N$}
    \For {$m = 1, \ldots, \delta$}
        \State $\widetilde{x^{(m)}(t_\ell)} \gets $ an approximation for $x^{(m)}(t_\ell)$ obtained from $x^{(m)}(t_{\ell-1})$ \label{line:predict}
        \State $x^{(m)}(t_\ell) \gets $ refine $\widetilde{x^{(m)}(t_\ell)}$ using Newton iteration \label{line:correct}
    \EndFor
    \State ${\rm Tr}_{{\cal A}_0}(I_{a^{(j)}})_{|z = \varphi(t_\ell)} = \left (\sum_{m = 1}^\delta (x^{(m)}(t_\ell))_i \right )_{i = 1, \ldots, n}$
    \State $\ell \gets \ell + 1$
\EndWhile
\end{algorithmic}
\end{algorithm}

\section{Complexity} \label{sec:5}
In this section, we discuss the complexity of the contour integration algorithm presented in Section \ref{sec:4}. We split the algorithm into two major steps: 
\begin{enumerate}
    \item Evaluate the moment matrices $A_0, \ldots, A_{2M-1}$.
    \item Extract the eigenvalues from these moment matrices.
\end{enumerate}

In Step 2, one constructs the matrices $B_0, B_1$ from \eqref{eq:Bi}. These are of size $M \cdot n$, and $M$ is chosen such that $M \cdot n \geq \delta(\Omega)$, where $\delta(\Omega)$ is the number of eigenvalues inside $\Omega$. The eigenvalues are then extracted from $B_0, B_1$ by computing an SVD, see Section \ref{sec:2}. The cost is $O(M^3 \cdot n^3)$. The most favorable situation for our method is when $\delta(\Omega) \approx M \cdot n \ll \hat{\delta}$. 

Step 1 uses numerical homotopy continuation. Continuing to work under Assumption \ref{assum:mv}, it requires tracking $n \cdot \delta = n \cdot {\rm MV}(P_1, \ldots, P_n)$ solution paths. The homotopy is used to evaluate $U(\varphi(t_\ell))$ as discussed in Section \ref{sec:4}. The moment matrices $A_k$ are then approximated via \eqref{eq:int_approx}. In our analysis, we assume that the number of nodes $N$ is fixed. Moreover, we ignore the complexity of computing $A_{k,N}$ from $U(\varphi(t_\ell))$, as it is negligible compared to the cost of tracking our $n \cdot \delta$ paths. 

The number $n \cdot \delta$ should be compared to the total number of eigenvalues of $T(x,z)$, denoted $\hat{\delta}$. This is the number of paths tracked in the naive approach of computing all eigenpairs and discarding those for which $z \notin \Omega$. However, we warn the reader that one cannot straightforwardly draw conclusions about the computation time by simply comparing $n \cdot \delta$ and $\hat{\delta}$. For instance, it might be favorable to solve $n$ problems with $\delta < \hat{\delta}$ solutions rather than one problem with $\hat{\delta}$ solutions, even if $n \cdot \delta > \hat{\delta}$.
Below, we compute the number of paths $n \cdot \delta$ for two families of PEPv's. The first one is inspired by Theorems \ref{thm:unmixed} and \ref{thm:pyramid}, where $d_1 = \cdots = d_n$. The second one is a family of systems of rational function equations from \cite{claes2022linearizable}, which can be solved using a slight modification of our method.

\subsection{Unmixed, dense equations}
We consider the case where $T(x,z) \cdot x = (f_1(x,z), \ldots, f_n(x,z))^\top$ comes from the polynomial system $f_1 = \cdots = f_n = 0$, where each $f_i$ is homogeneous of degree $d+1$ in $x$, and of degree $e$ in $z$. We assume that $x_j^{d+1}, j = 1, \ldots, n$ appear in each of the $f_i$. First, we also choose the polynomials $a_i(x) \in \C[x]_d$ such that $x_j^d, j = 1, \ldots, n$ appear in each of them. This is the situation of Theorem \ref{thm:unmixed}. We compute the numbers $n \cdot \delta$ and $\hat{\delta}$ for this setup. 

\begin{proposition} \label{prop:rat1}
Let $f_1, \ldots, f_n, a_1, \ldots, a_n$ be as in Theorem \ref{thm:unmixed}. We have
\[ n \cdot \delta = n \cdot ((d+1)^n - d^n), \quad \hat{\delta} = e \cdot n \cdot  (d+1)^{n-1}.\]
\end{proposition}
\begin{proof}
By the multihomogeneous version of B\'ezout's theorem, the total number of eigenvalues, i.e., solutions to $f_1 = \cdots = f_n = 0$, is $\hat{\delta} = e \cdot n \cdot (d+1)^{n-1}$. To compute $\delta$, consider the polytope $P = P_1 = \cdots = P_n \subset \R^n$, given by \eqref{eq:dense}, with $d_i = d$. By Kushnirenko's theorem, the number $\delta$ is the lattice volume of $P$. This is given by $\delta = (d+1)^n - d^n$. 
\end{proof}
It follows that, for large $d$, the ratio $(n \cdot \delta)/\hat{\delta}$ tends to $n/e$. Hence, our method tracks significantly fewer solution paths when $e \gg n$. We note that, for small $d$, this conclusion is pessimistic. For instance, if $d = 2$, we find that $(n \cdot \delta)/\hat{\delta} \approx 2/e$. 

A smaller number of paths $n \cdot \delta$ is obtained when the $a_i(x)$ are chosen as in Theorem \ref{thm:pyramid}. The computation is similar to the proof of Proposition \ref{prop:rat1}, noting that the lattice volume of a pyramid of lattice height 1 equals the $(n-1)$-dimensional lattice volume of its base.
\begin{proposition} \label{prop:rat2}
Let $f_1, \ldots, f_n, a_1, \ldots, a_n$ be as in Theorem \ref{thm:pyramid}. We have
\[ n \cdot \delta =  n \cdot (d+1)^{n-1}.\]
\end{proposition}
Propositions \ref{prop:rat1} and \ref{prop:rat2} lead us to conclude that the methods presented in this paper are effective only when the degree in the eigenvalue variable is large. This situation arises, for instance, when the PEPv comes from a polynomial approximation of a set of equations that depends transcendentally on $z$. We will show an example in Section \ref{exp:3}.
\subsection{Rational functions}
We now discuss an example where the entries of the matrix $T(x,z)$ are homogeneous \emph{rational} functions in $x$. More precisely, consider a rational map $T: \PP^{n-1} \times \C \dashrightarrow \C^{n \times n}$ of the form
\begin{equation} \label{eq:rationalT}
    T(x,z) \, = \, T_0(z) + \frac{r_1(x)}{s_1(x)} \, T_1 \, + \, \cdots \, + \, \frac{r_m(x)}{s_m(x)} \,T_m, 
\end{equation} 
where $T_0(z) = A + z \cdot B$ with $A, B \in \C^{n \times n}$, and $r_i(x), s_i(x)$ are linear forms in $x$. The associated \emph{rational eigenvalue problem with eigenvector nonlinearities} (REPv) is 
\begin{equation} \label{eq:ratprob} \text{find $(x^*,z^*) \in (\PP^{n-1} \setminus V_{\PP^{n-1}}(s_1 \cdots s_m)) \times \C$ such that $T(x^*,z^*) \cdot x^* = 0$.}
\end{equation}
Here we use the standard notation $V_X(f) = \{ x \in X ~|~ f(x) = 0 \}$. The problem \eqref{eq:ratprob} was studied in \cite{claes2022linearizable}. We here discuss how our methods can be used to solve this REPv. We point out that, in this case, the problem cannot be turned into a PEPv by clearing denominators, as this typically introduces infinitely many spurious eigenvectors.

The rows of $T$ are homogeneous of degree $d = 0$ in $x$. Consistently with our approach for PEPv's, we consider the equations $T(x,z) \cdot x - a = (f_1 - a_1, \ldots, f_n - a_n)^\top =  0$, where $a = (a_1, \ldots, a_n)^\top \in \C^n$ is a generic vector of complex constants. The matrix $U(z)$ from \eqref{eq:U} is constructed by summing over the $\delta$ solutions. The following theorem predicts $\delta$.

\begin{theorem} \label{thm:nosolsrat}
For $T$ as in \eqref{eq:rationalT} and generic $z \in \C, a \in \C^n$, the system of equations $T(x,z) \cdot x - a = 0$ has at most $\delta$ isolated solutions in $(\C^n \setminus V_{\C^n}(s_1 \cdots s_m)) \times \C$, with 
\[ \delta = \sum_{k=0}^{\min(n-1,m)} \begin{pmatrix} n-1 \\ k
\end{pmatrix} \cdot \begin{pmatrix} m \\ k
\end{pmatrix}. \]
\end{theorem}
\begin{proof}[Sketch of proof]
The system of rational function equations $T(x,z) \cdot x - a = 0$ is equivalent to the system of $n + m$ polynomial equations
\begin{equation}  \label{eq:rattopol}
(T_0 + \lambda_1 T_1 + \cdots + \lambda_n T_n) \cdot x - a = 0, \quad s_i(x) \lambda_i - r_i(x) = 0, i = 1, \ldots, m, \end{equation}
where $\lambda_1, \ldots, \lambda_m$ are new variables and $T_0 = T_0(z)$. The entries of $(T_0 + \lambda_1 T_1 + \cdots + \lambda_n T_n) \cdot x - a$ all have the same Newton polytope, denoted $P \subset \R^{m+n}$. The equation $s_i(x)\lambda_i - r_i(x)$ has Newton polytope $\Delta_n \times L_i$, where $\Delta_n = {\rm Conv}(e_1, \ldots, e_n) \subset \R^n$ and $L_i = {\rm Conv}(0,e_i) \subset \R^m$. By the BKK theorem, the number of isolated solutions to \eqref{eq:rattopol} is bounded by the mixed volume $\delta ={\rm MV}(P, \ldots, P, \Delta_n \times L_1, \ldots, \Delta_n \times L_m)$. Here $P$ is listed $n$ times. Multilinearity and symmetry of the mixed volume gives the equality
\[ \delta = \sum_{k=0}^m \begin{pmatrix}
m \\ k
\end{pmatrix}{\rm MV}(P,\ldots, P, \Delta_n, \ldots, \Delta_n, L_{k+1}, \ldots, L_m). \]
Since $\Delta_n$ has dimension $n-1$, all terms with $k > n-1$ are zero. It remains to show that for $k \leq \min(n-1,m)$, we have 
\[{\rm MV}(P,\ldots, P, \Delta_n, \ldots, \Delta_n, L_{k+1}, \ldots, L_m) = \begin{pmatrix}
n-1 \\ k
\end{pmatrix}.\]
This number counts solutions to $(T_0 + \lambda_1 T_1 + \cdots + \lambda_n T_n) \cdot x - a = 0$ after plugging in random values for $\lambda_{k+1}, \ldots, \lambda_m$ and replacing $x_{n-k+1}, \ldots, x_n$ by generic linear forms in $x_1, \ldots, x_{n-k}$. What is left is a system of $n$ equations in the variables $(x_1, \ldots, x_{n-k}, \lambda_1, \ldots, \lambda_k)$. It has at most $\begin{pmatrix}
n-1 \\ k
\end{pmatrix}$ solutions by the multihomogeneous version of B\'ezout's theorem.
\end{proof}
By \cite[Theorem 3.1]{claes2022linearizable}, the total number of eigenvalues of \eqref{eq:rationalT} is $\hat{\delta} = \begin{pmatrix} n +m \\ m+1
\end{pmatrix}$. Although $\delta < \hat{\delta}$, we have $n \cdot \delta > \hat{\delta}$. We will illustrate this with an example in Section \ref{sec:6}. We leave the question whether and when our method is advantageous for solving this type of REPv as a topic for future research.

\section{Numerical experiments} \label{sec:6}

In this section we present several numerical examples illustrating the results presented above. Our algorithm has two important parameters that impact the numerical performance: the number $N+1$ of discretization points on the contour to evaluate the integral, and the number of moment matrices $2M$.
In the experiments below, we will investigate the influence of these parameters on the accuracy. We assess the quality of an approximate eigenpair $(x^*,z^*)$ by its residual $r^* = \lVert T(x^*,z^*)\cdot x^* \rVert/\lVert x^*\rVert$.
The presented result are generated by an implementation in Julia (v1.6) using \texttt{HomotopyContinuation.jl} (v2.6.4) \cite{breiding2018homotopy}.
The source code is available online to reproduce all results\footnote{\texttt{github.com/robclaes/contour-integration}}.

\subsection{Experiment 1}

Consider the PEPv $T(x,z)\cdot x=0$ where $T(x,z)$ has size $3\times 3$ and each row is of degree $d = 2$ in $x$ and $e = 4$ in $z$.
The coefficients are randomly generated in order to obtain a generic system.
The contour enclosing the target domain $\Omega$ is shown in Figure~\ref{fig:exp1} together with the exact eigenvalues in the neighborhood of $\Omega$.

The impact of the number of discretization points $N+1$ is the most intuitive: the more points, the higher the accuracy of the detected eigenvalues in $\Omega$.
There is a less intuitive impact that has been observed in contour integration for nonlinear eigenvalue problems \cite{van2016nonlinear}.
When the contour integral is approximated with a low number of points, it is possible that eigenvalues outside the contour are detected.
Evaluating the contour integral with $1000$ points detects only the four eigenvalues in $\Omega$ with average residual in the order of magnitude of machine precision.
However, evaluating the contour integral with $100$ points, detects 14 eigenvalues depicted in Figure~\ref{fig:exp1}: four eigenvalues in $\Omega$ with average residual of $\approx 10^{-11}$ and eight eigenvalues outside the target domain with residual varying from $10^{-9}$ to $10^{-5}$ depending on the distance from the contour. This phenomenon is best explained via the relation between numerical integration and filter functions on $\C$, see \cite{van2016nonlinear} for details.

An obvious impact of the number of moment matrices can be seen in \eqref{eq:Bi}: the maximum number of eigenvalues that can be detected is $Mn$.
Therefore $M$ should be large enough to detect at least the expected number of eigenvalues in $\Omega$.
However when a low number of discretization points is chosen, extra care must be taken when choosing the number of moment matrices: the algorithm will detect additional eigenvalues outside $\Omega$ which may lead to more eigenvalues than the number of eigenvalues that can be detected for a given $M$. 
For the specific instance here, we selected $M=9$ which leads to a maximum of $Mn=27$ detectable eigenvalues.
In the case with $100$ discretization points this upper bound is large enough to detect the 14 eigenvalues.
When we set $M=2$ -- which should suffice for the expected $4$ eigenvalues in $\Omega$ -- with $100$ discretization points, the eigenvalues outside $\Omega$ perturb the result leading the an average residual of the 4 eigenvalues in $\Omega$ of $10^{-3}$.

Since the degree of the polynomials is the same for each row, we select the polynomials $a_i$ in accordance with Theorem~\ref{thm:pyramid}, i.e., $a_i$ is a monomial in $x$ of degree $d=2$.
By Proposition \ref{prop:rat2}, this leads to $n\cdot\delta =n\cdot(d+1)^{n-1}=27$ tracked paths, which is smaller than the expected number of tracked paths when using random polynomials: $n\cdot\delta=n\cdot\left((d+1)^n-d^n\right)=57$.

\begin{figure}
    \centering
    \input{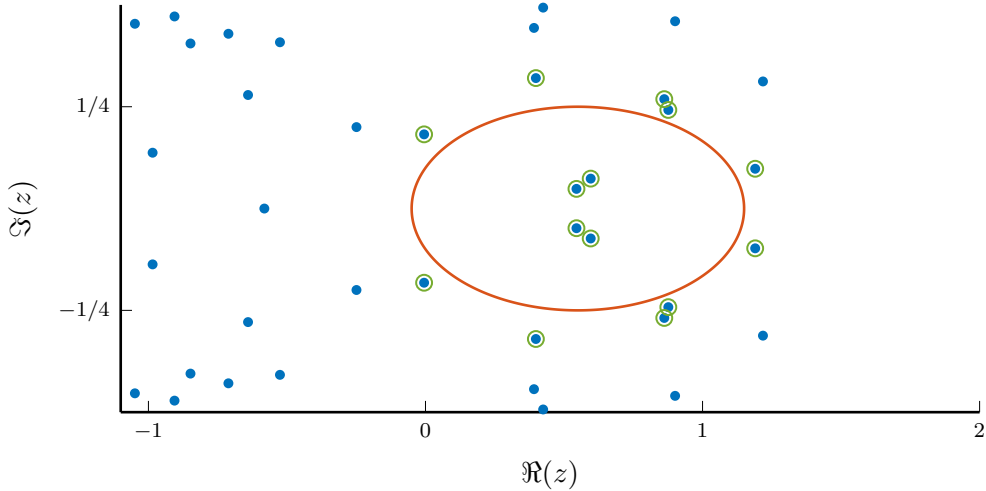}
    \caption{Eigenvalues (\ref{marker:eigs}) inside the target domain defined by the contour (\ref{marker:contour}) and the extracted values by contour integration (\ref{marker:CI}) for experiment 1.}
    \label{fig:exp1}
\end{figure}

\subsection{Experiment 2}
Consider the PEPv $T(x,z)\cdot x=0$ where $T(x,z)$ has size $10\times 10$ and each row is of degree $d = 1$ in $x$ and $e = 5$ in $z$.
The coefficients are randomly generated.
The contour enclosing the target domain $\Omega$ is shown in Figure~\ref{fig:exp2} together with the exact eigenvalues in the neighborhood of $\Omega$.
This is a very nontrivial problem since the total number of solutions of the PEPv equals $\hat{\delta}=25600$ and they are almost all clustered around the origin of the complex plane.
The selected contour is a circle with center at the origin and a radius of $0.1$ which encircles $44$ eigenvalues of the problem.

Since the neighborhood of the target region $\Omega$ is densely scattered with eigenvalues, we select a relatively high number of integration points $N+1=400$ to increase the sharpness of the integration filter as discussed in the previous example.
Given the high number of integration points, a maximum of $2M=10$ moment matrices should suffice to capture the $44$ expected eigenvalues in $\Omega$.
The result is shown in Figure~\ref{fig:exp2}: a total of $46$ detected eigenvalues: $44$ inside $\Omega$ and $2$ just outside the target region.
The residual for the extracted eigenpairs varies from $10^{-4}$ to $10^{-8}$.
In accordance with Theorem~\ref{thm:pyramid}, we selected $a_i$ as a monomial of degree $d=1$ which leads to $n\cdot\delta=n\cdot(d+1)^{n-1}=5120$ tracked paths.
Therefore, finding all solutions with standard homotopy continuation takes roughly 2390 seconds to compute, while our approach with $400$ interpolation points takes 1120 seconds.
(Both timings result from a single-thread implementation in Julia).

\begin{figure}
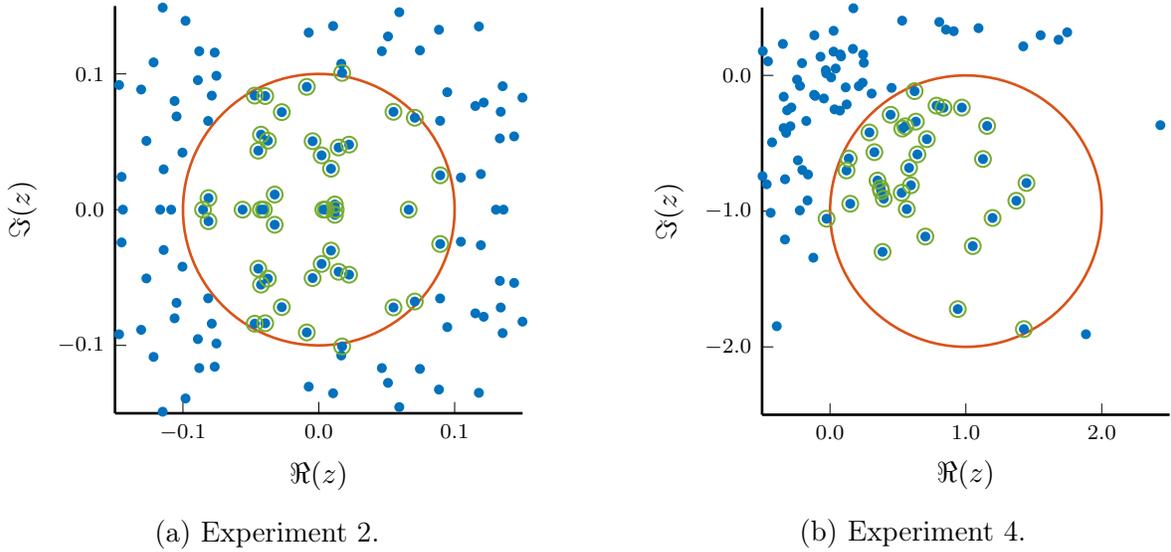

    \centering
    \begin{subfigure}{0.48\textwidth}
        \centering
        \input{experiment2}
        \caption{Experiment 2.}
        \label{fig:exp2}
    \end{subfigure}
    \hfill
    \begin{subfigure}{0.48\textwidth}
        \centering
        \input{experiment4}
        \caption{Experiment 4.}
        \label{fig:exp4}
    \end{subfigure}
    \caption{Eigenvalues (\ref{marker:eigs}) inside the target domain defined by the contour (\ref{marker:contour}) and the extracted values by contour integration (\ref{marker:CI}).}
    \label{fig:my_label}
\end{figure}

\subsection{Experiment 3}
\label{exp:3}
Consider the system of equations $T(x,z)\cdot x=0$ given by
\begin{equation}
    T(x,z) = \begin{pmatrix}
    x_1^2x_2 & -2\sqrt{-1}x_1^2x_2\cos(z)\\
    -x_2^2\cos(z^2) & 2x_2^2\sin(3z)
    \end{pmatrix}.
\end{equation}
Note that this system is not polynomial in $z$, but in practice the system is solved by an implicit substitution of Maclaurin series of high order for the sine and cosine functions. 
This approach leads to a PEPv that is of high degree in $z$.
We expect an infinite number of solutions since the trigonometric functions can be expressed by their Maclaurin series in $z$.
We use $100$ discretization points for the contour, and $2M = 16$ moment matrices.
The $a_i$ are selected as random monomials in $x$ that have the same degree as the polynomials in the corresponding row of $T(x,z)$, similarly as in Theorem~\ref{thm:pyramid}.
This leads to 4 tracked paths, instead of 10 for random polynomials.
Figure~\ref{fig:exp3_conv} shows the impact of the number of discretization points on the residual of the 11 extracted solutions.
As stated in experiment 1, increasing the number of discretization points leads to a decrease in the residual.

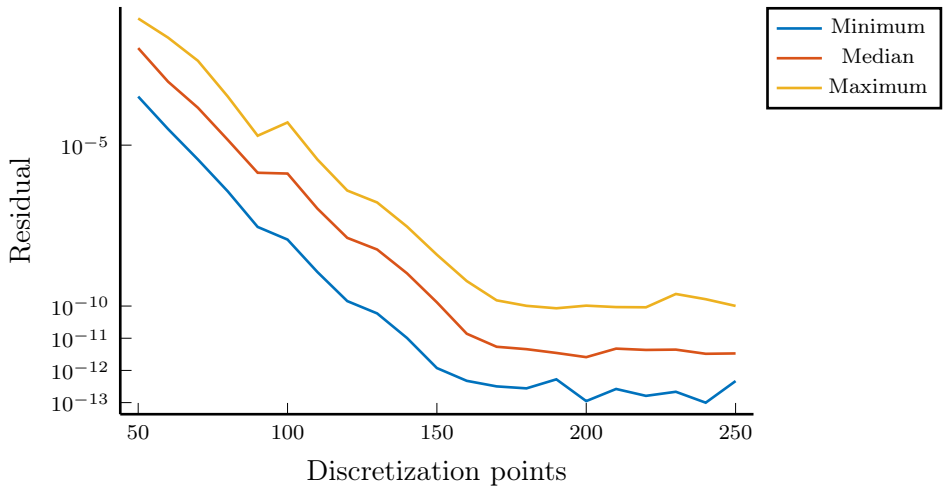
\begin{figure}
    \centering
    \begin{tikzpicture}[/tikz/background rectangle/.style={fill={rgb,1:red,1.0;green,1.0;blue,1.0}, draw opacity={1.0}}]
\begin{axis}[point meta max={nan}, point meta min={nan}, legend cell align={left}, legend columns={1}, title={}, title style={at={{(0.5,1)}}, anchor={south}, font={{\fontsize{14 pt}{18.2 pt}\selectfont}}, color={rgb,1:red,0.0;green,0.0;blue,0.0}, draw opacity={1.0}, rotate={0.0}}, legend style={color={rgb,1:red,0.0;green,0.0;blue,0.0}, draw opacity={1.0}, line width={1}, solid, fill={rgb,1:red,1.0;green,1.0;blue,1.0}, fill opacity={1.0}, text opacity={1.0}, font={{\fontsize{8 pt}{10.4 pt}\selectfont}}, text={rgb,1:red,0.0;green,0.0;blue,0.0}, cells={anchor={center}}, at={(1.02, 1)}, anchor={north west}}, axis background/.style={fill={rgb,1:red,1.0;green,1.0;blue,1.0}, opacity={1.0}}, anchor={north west}, xshift={1.0mm}, yshift={-1.0mm}, width={100mm}, height={70mm}, scaled x ticks={false}, xlabel={Discretization points}, x tick style={color={rgb,1:red,0.0;green,0.0;blue,0.0}, opacity={1.0}}, x tick label style={color={rgb,1:red,0.0;green,0.0;blue,0.0}, opacity={1.0}, rotate={0}}, xlabel style={at={(ticklabel cs:0.5)}, anchor=near ticklabel, at={{(ticklabel cs:0.5)}}, anchor={near ticklabel}, font={{\fontsize{11 pt}{14.3 pt}\selectfont}}, color={rgb,1:red,0.0;green,0.0;blue,0.0}, draw opacity={1.0}, rotate={0.0}}, xmin={44.0}, xmax={256.0}, xtick={{50.0,100.0,150.0,200.0,250.0}}, xticklabels={{$50$,$100$,$150$,$200$,$250$}}, xtick align={inside}, xticklabel style={font={{\fontsize{8 pt}{10.4 pt}\selectfont}}, color={rgb,1:red,0.0;green,0.0;blue,0.0}, draw opacity={1.0}, rotate={0.0}}, x grid style={color={rgb,1:red,0.0;green,0.0;blue,0.0}, draw opacity={0.1}, line width={0.5}, solid}, axis x line*={left}, x axis line style={color={rgb,1:red,0.0;green,0.0;blue,0.0}, draw opacity={1.0}, line width={1}, solid}, scaled y ticks={false}, ylabel={Residual}, y tick style={color={rgb,1:red,0.0;green,0.0;blue,0.0}, opacity={1.0}}, y tick label style={color={rgb,1:red,0.0;green,0.0;blue,0.0}, opacity={1.0}, rotate={0}}, ylabel style={at={(ticklabel cs:0.5)}, anchor=near ticklabel, at={{(ticklabel cs:0.5)}}, anchor={near ticklabel}, font={{\fontsize{11 pt}{14.3 pt}\selectfont}}, color={rgb,1:red,0.0;green,0.0;blue,0.0}, draw opacity={1.0}, rotate={0.0}}, ymode={log}, log basis y={10}, ymin={4.3581874493397274e-14}, ymax={0.19669137868571174}, ytick={{1.0e-5,1.0e-10,1.0e-11,1.0e-12,1.0e-13}}, yticklabels={{$10^{-5}$,$10^{-10}$,$10^{-11}$,$10^{-12}$,$10^{-13}$}}, ytick align={inside}, yticklabel style={font={{\fontsize{8 pt}{10.4 pt}\selectfont}}, color={rgb,1:red,0.0;green,0.0;blue,0.0}, draw opacity={1.0}, rotate={0.0}}, y grid style={color={rgb,1:red,0.0;green,0.0;blue,0.0}, draw opacity={0.1}, line width={0.5}, solid}, axis y line*={left}, y axis line style={color={rgb,1:red,0.0;green,0.0;blue,0.0}, draw opacity={1.0}, line width={1}, solid}, colorbar={false}]
    \addplot[color=mycolor1,  draw opacity={1.0}, line width={1}, solid]
        table[row sep={\\}]
        {
            \\
            50.0  0.00032297556867893134  \\
            60.0  3.165767621366915e-5  \\
            70.0  3.5799428139726276e-6  \\
            80.0  3.6561079633110145e-7  \\
            90.0  2.8724664485481023e-8  \\
            100.0  1.1627380884341356e-8  \\
            110.0  1.1424158897092236e-9  \\
            120.0  1.4253430229282738e-10  \\
            130.0  5.90786517022531e-11  \\
            140.0  1.0174503599032456e-11  \\
            150.0  1.184195974893719e-12  \\
            160.0  4.74784027858908e-13  \\
            170.0  3.1907982724276905e-13  \\
            180.0  2.77159094300411e-13  \\
            190.0  5.285471971650926e-13  \\
            200.0  1.1144981208101681e-13  \\
            210.0  2.6527106649798715e-13  \\
            220.0  1.6207025163484862e-13  \\
            230.0  2.170571980076027e-13  \\
            240.0  9.941491593564362e-14  \\
            250.0  4.666300013500311e-13  \\
        }
        ;
    \addlegendentry {Minimum}
    \addplot[color=mycolor2,draw opacity={1.0}, line width={1}, solid]
        table[row sep={\\}]
        {
            \\
            50.0  0.010356128768253068  \\
            60.0  0.000938409050477471  \\
            70.0  0.00014512139700793267  \\
            80.0  1.455329848457706e-5  \\
            90.0  1.385793605553641e-6  \\
            100.0  1.3094984125726255e-6  \\
            110.0  1.0675802955264307e-7  \\
            120.0  1.3145436385831498e-8  \\
            130.0  5.7141204654576766e-9  \\
            140.0  1.0384263609228616e-9  \\
            150.0  1.304610385622612e-10  \\
            160.0  1.380591839215439e-11  \\
            170.0  5.4320789078795914e-12  \\
            180.0  4.5724007301143065e-12  \\
            190.0  3.490898810968369e-12  \\
            200.0  2.585484738003879e-12  \\
            210.0  4.756769373057153e-12  \\
            220.0  4.331855437300488e-12  \\
            230.0  4.4294070049551e-12  \\
            240.0  3.2872505921898884e-12  \\
            250.0  3.3715501519623125e-12  \\
        }
        ;
    \addlegendentry {Median}
    \addplot[color=mycolor3,  draw opacity={1.0}, line width={1}, solid]
        table[row sep={\\}]
        {
            \\
            50.0  0.08622628605714612  \\
            60.0  0.021917403372299683  \\
            70.0  0.0041724977108685835  \\
            80.0  0.00032352652909829115  \\
            90.0  1.9612073043114425e-5  \\
            100.0  5.069184925829584e-5  \\
            110.0  3.586445287826807e-6  \\
            120.0  3.8820371462841895e-7  \\
            130.0  1.655688903194703e-7  \\
            140.0  2.9440433679053535e-8  \\
            150.0  3.935847861486385e-9  \\
            160.0  5.982899533133101e-10  \\
            170.0  1.5058813163731429e-10  \\
            180.0  1.0166747056029275e-10  \\
            190.0  8.526763291507809e-11  \\
            200.0  1.0293699726750942e-10  \\
            210.0  9.288689513868011e-11  \\
            220.0  9.1510279521014e-11  \\
            230.0  2.373574299500802e-10  \\
            240.0  1.63135765754764e-10  \\
            250.0  1.0082795853451223e-10  \\
        }
        ;
    \addlegendentry {Maximum}
\end{axis}
\end{tikzpicture}
    \caption{Impact of number of discretization points on residuals for experiment~3.}
    \label{fig:exp3_conv}
\end{figure}

\subsection{Experiment 4}
Consider the REPv \eqref{eq:rationalT} of dimension $n=10$ with $m=2$ rational terms where all coefficients are randomly generated.
A problem with these dimensions is expected to have $\hat{\delta} =\begin{pmatrix} n +m \\ m+1
\end{pmatrix}=220$ eigenvalues.
According to Theorem~\ref{thm:nosolsrat} we need to track $n\cdot\delta=550$ paths.
As depicted in Figure~\ref{fig:exp4}, all $33$ eigenvalues in the contour are detected with a residual ranging from $10^{-8}$ to $10^{-12}$, and one eigenvalue outside of the contour with a residual of $10^{-6}$.
This result is obtained using $N+1=400$ nodes, and $2M=10$ moment matrices.

\section{Conclusions}
We presented a new contour integration method for solving polynomial eigenvalue problems with eigenvector nonlinearities and developed its first theoretical foundations.
The eigenvalues are the roots of a resultant polynomial.
We showed that, under suitable assumptions, this polynomial equals the denominator of the trace obtained by summing over the solutions to a modified system of equations. This can be evaluated along a contour using numerical homotopy continuation techniques. 
This way, we can extract eigenvalues in a compact domain and their corresponding eigenvectors by numerical contour integration.
We derived the number of homotopy continuation paths that need to be tracked for two classes of problems.
This governs, to a certain extent, the complexity of our method. However, a direct comparison with the total number of eigenvalues is not very meaningful since the difficulty and computational cost of tracking a single path may differ greatly.
A comparative study on the total computational cost is an interesting topic for future research, together with a study on the applicability of other NEP methods on the compound trace matrix $U(z)$.

\section*{Acknowledgements} 
The work by Rob Claes and Karl Meerbergen is supported by the Research Foundation Flanders (FWO) Grant G0B7818N and the KU Leuven Research Council.

We would like to thank Paul Breiding for his help with \texttt{HomotopyContinuation.jl}, and Carlos D'Andrea for insightful discussions. 

\bibliographystyle{abbrv}
\bibliography{references.bib} 
\noindent
\footnotesize 
{\bf Authors' addresses:}

\smallskip

\noindent Rob Claes, KU Leuven
\hfill {\tt rob.claes@kuleuven.be}

\noindent Karl Meerbergen, KU Leuven
\hfill {\tt karl.meerbergen@kuleuven.be}

\noindent Simon Telen, MPI-MiS Leipzig
\hfill {\tt simon.telen@mis.mpg.de}

\end{document}